\documentclass[reqno,11pt]{amsart}
\usepackage{amsmath,amssymb,amsthm,graphicx,a4wide,enumerate,url}
\usepackage[small,bf]{caption}
\theoremstyle{plain}
\newtheorem{thm}{Theorem}
\newtheorem{lemma}[thm]{Lemma}

\theoremstyle{definition}

\theoremstyle{remark}
\newtheorem{rem}{Remark}

\newcommand{\prn}[1]{\left(#1\right)}

\newcommand{\BRK}[1]{\left<#1\right>}

\newcommand{\ud}[1]{\,\mathrm{d}#1}
\renewcommand{\vec}[1]{\mathbf{#1}}
\makeatletter
\renewcommand{\env@cases}[1][@{}l@{\quad}l@{}]{%
  \let\@ifnextchar\new@ifnextchar
  \left\lbrace
  \def\arraystretch{1.2}%
  \array{#1}%
}
\makeatother

\begin{document}
\parskip.9ex

\title[Meshfree FD for Vector Poisson Equations with Electric Boundary Conditions]
{Meshfree Finite Differences for Vector Poisson and Pressure Poisson Equations with Electric Boundary Conditions}

\author[D. Zhou]{Dong Zhou}
\address[Dong Zhou]
{Department of Mathematics \\ Temple University \\ \newline
1805 North Broad Street \\ Philadelphia, PA 19122}
\email{dong.zhou@temple.edu}
\urladdr{http://www.math.temple.edu/\~{}dzhou}

\author[B. Seibold]{Benjamin Seibold}
\address[Benjamin Seibold]
{Department of Mathematics \\ Temple University \\ \newline
1805 North Broad Street \\ Philadelphia, PA 19122}
\email{seibold@temple.edu}
\urladdr{http://www.math.temple.edu/\~{}seibold}

\author[D. Shirokoff]{David Shirokoff}
\address[David Shirokoff]
{Department of Mathematics and Statistics \\ McGill University \\
\newline 805 Sherbrooke Street West \\ Montreal, Quebec \\ Canada, H3A 0B9}
\email{david.shirokoff@mail.mcgill.ca}
\urladdr{http://www.math.mcgill.ca/dshirokoff}

\author[P. Chidyagwai]{Prince Chidyagwai}
\address[Prince Chidyagwai]
{Department of Mathematics and Statistics \\ Loyola University Maryland \\ \newline
4501 N. Charles Street \\ Baltimore, MD 21210}
\email{pchidyagwai@loyola.edu}
\urladdr{http://math.loyola.edu/\~{}chidyagp}

\author{Rodolfo Ruben Rosales}
\address[Rodolfo Ruben Rosales]
{Department of Mathematics \\ Massachusetts Institute of Technology \\
\newline 77 Massachusetts Avenue \\ Cambridge, MA 02139}
\email{rrr@math.mit.edu}





\begin{abstract}
We demonstrate how meshfree finite difference methods can be applied to solve vector Poisson problems with electric boundary conditions. In these, the tangential velocity and the incompressibility of the vector field are prescribed at the boundary. Even on irregular domains with only convex corners, canonical nodal-based finite elements may converge to the wrong solution due to a version of the Babu{\v s}ka paradox. In turn, straightforward meshfree finite differences converge to the true solution, and even high-order accuracy can be achieved in a simple fashion. The methodology is then extended to a specific pressure Poisson equation reformulation of the Navier-Stokes equations that possesses the same type of boundary conditions. The resulting numerical approach is second order accurate and allows for a simple switching between an explicit and implicit treatment of the viscosity terms.
\end{abstract}

\maketitle

\section{Introduction}
\label{sec:introduction}
The numerical approximation of vector fields that are incompressible is often times challenging because incompressibility, $\nabla\cdot\vec{u} = 0$, is a global constraint that may not fit within the framework of simple discretization approaches of the complete problem. The instationary incompressible Navier-Stokes equations (NSE) represent a prime example, in which the time-evolution of the velocity field is given, but not of the pressure (which is a Lagrange multiplier associated with $\nabla\cdot\vec{u} = 0$). As a consequence, there is no single canonical way to advance the NSE forward in time. Similarly, in electrostatics the electric field and the magnetic potential are solutions to vector Poisson equations with divergence constraints.

One methodology to circumvent the incompressibility constraint inside the computational domain is to formulate a different problem that imposes $\nabla\cdot\vec{u} = 0$ as a boundary condition instead. Under certain circumstances, this new problem has the same solution as the original problem, while at the same time giving rise to new numerical approximation methods. For electrostatic problems (see \S\ref{sec:vector_poisson_equation}) this approach is employed and analyzed in \cite{Mayergoyz1993, JiangWuPovinelli1996}, and in the context of incompressible fluid flows (see \S\ref{sec:PPE_reformulation}) it has been proposed in \cite{ShirokoffRosales2010}. The specific boundary conditions for these problems consist of enforcing the tangential component(s) of the solution, together with the condition $\nabla\cdot\vec{u} = 0$. Due to their occurrence in electrostatics, they are often called (perfect) \emph{electric boundary conditions} (EBC), a terminology that we follow in this paper.

A fundamental question is what are simple numerical approaches to approximate the solutions to vector-valued problems with EBC. In \cite{ShirokoffRosales2010} an immersed boundary staggered grid approach has been proposed. While convergent, this approach is definitely not a simple or canonical method. A seemingly more natural approach for problems on irregular domains is the standard nodal finite element method (FEM). Interestingly, for the class of problems at hand, nodal FEM can exhibit a Babu\u{s}ka paradox, i.e., for domains with curved boundaries the sequence of FEM approximations can converge to a wrong solution (see \S\ref{subsec:nodal_FEM}). Domains with re-entrant corners (see \S\ref{subsec:EBC}) pose additional challenges; however, this last aspect is not the focus of this paper. The convergence of the FEM can be recovered by converting to a mixed FEM formulation, as conducted in \cite{ChidyagwaiRosalesSeiboldShirokoffZhou2013}. However, the simplicity of nodal FEM does not carry over to mixed FEM.

In this paper, we present yet another alternative for problems with EBC, namely meshfree finite differences (FD). These are generalizations of traditional grid-based FD that apply to clouds of points without any connectivity between them. The meshfree FD methodology is presented in \S\ref{subsec:VPE_meshfree_FD} and \S\ref{subsec:point_cloud_generation}, and its application to the vector Poisson equation with EBC is shown in \S\ref{subsec:VPE_numerical_results}. Structurally, the approximation of a general PDE boundary value problem via meshfree FD is very straightforward: any differential operator, whether in the domain's interior or on the boundary, is approximated via a meshfree FD stencil. Thus, the problem is directly transformed into a finite dimensional system, in which each individual equation corresponds to the governing condition that holds at a particular point of the cloud.

We demonstrate (in \S\ref{subsec:VPE_numerical_results}) that for the vector Poisson equation with EBC, meshfree finite differences do not exhibit the Babu\u{s}ka paradox, and furthermore that there is no conceptual problem to obtain higher-order accuracy (we test the method up to third order convergence).

We then move on to time-dependent problems. First, the meshfree FD method is extended to the vector heat equation with EBC (see \S\ref{subsec:VHE_meshfree_FD}), where an explicit or an implicit time-stepping can be conducted. Then, by adding nonlinear convective terms and a pressure, the approach is further extended to the incompressible Navier-Stokes equations (NSE). Specifically, we consider a pressure Poisson equation (PPE) reformulation of the Navier-Stokes equations. The idea of PPE reformulations (see \S\ref{sec:PPE_reformulation}) is that an operator function $p = P(\vec{u})$ is formulated that yields (via the solution of a Poisson equation) the pressure $p$ to any given velocity field $\vec{u}$ that solves the NSE. Here, we focus on a specific PPE reformulation, proposed in \cite{ShirokoffRosales2010}, which prescribes EBC for the fluid velocity, the motivation for which is outlined in \S\ref{subsec:PPE_EBC}. We demonstrate how a meshfree FD approximation for the full PPE reformulation can be constructed (see \S\ref{subsec:PPE_meshfree_FD}), and show computational results for a resulting numerical scheme that is second order accurate in space and time, and that allows for a choice of an explicit or an implicit treatment of the viscosity (see \S\ref{subsec:PPE_numerical_results}).

\vspace{1.5em}
\section{Vector Poisson Equation}
\label{sec:vector_poisson_equation}
The vector Poisson equation (VPE) arises, for instance, in problems in electrostatics. The electric field satisfies $\nabla\cdot\vec{E} = \rho$, where $\rho = \rho(x)$ is the (normalized) charge density. Using the fact that $\nabla\times\vec{E} = 0$, this implies the VPE $\Delta\vec{E} = \nabla\rho$. Moreover, if the boundaries of the domain are perfect conductors, then the vector field is perpendicular to the boundary, i.e., $\vec{n}\times\vec{E} = 0$, where $\vec{n}$ is the outer surface normal vector. Another example is the magnetic potential, which satisfies the VPE $\Delta\vec{A} = -\vec{J}$, where $\vec{J} = \vec{J}(x)$ is the (normalized) electric current density. The Coulomb gauge yields $\nabla\cdot\vec{A} = 0$, and the boundary condition $\vec{n}\times\vec{A} = 0$ represents a zero magnetic magnetic field (see \cite{ElDabaghiPironneau1986} for more details). Motivated by the structure of these examples, we here consider the VPE
\begin{equation}
\label{eq:VPE_original}
\begin{cases}[r@{~}l@{\quad}l]
-\Delta\vec{u} &= \vec{f} &\text{in~}\Omega \\
\nabla\cdot\vec{u} &= 0 &\text{on~}\Omega \\
\vec{n}\times\vec{u} &= \vec{n}\times\vec{g} &\text{on~}\partial\Omega\;.
\end{cases}
\end{equation}
As motivated in \S\ref{sec:introduction}, it can be desirable to remove the divergence condition that holds in the whole domain. In the following, we outline how this can be achieved.

\subsection{Electric Boundary Conditions}
\label{subsec:EBC}
Let $\Omega$ be a bounded, simply connected domain with Lipschitz boundary $\partial\Omega$. Moreover, in this paper we restrict to domains with boundaries $\partial\Omega$ which are piecewise $C^2$ and convex (see Remark~\ref{rem:re-entrant_corners}). We denote by $\vec{n}$ the outward unit normal vector along the boundary (that is defined almost everywhere). The vector Poisson equation (VPE) with electric boundary conditions (EBC) takes the form
\begin{equation}
\label{eq:VPE_EBC}
\begin{cases}[r@{~}l@{\quad}l]
-\Delta\vec{u} &= \vec{f} &\text{in~}\Omega \\
\nabla\cdot\vec{u} &= 0 &\text{on~}\partial\Omega \\
\vec{n}\times\vec{u} &= \vec{n}\times\vec{g} &\text{on~}\partial\Omega
\end{cases}
\end{equation}
where the source is incompressible, i.e., $\nabla\cdot\vec{f} = 0$. Note that in contrast to problem \eqref{eq:VPE_original}, problem \eqref{eq:VPE_EBC} possesses no source-free condition in the domain's interior. Instead, $\nabla\cdot\vec{u} = 0$ is specified as an additional boundary condition. Clearly, any solution of \eqref{eq:VPE_original} is also a solution of \eqref{eq:VPE_EBC}. Moreover\dots
\begin{lemma}
\label{lem:div-free}
If the solution to \eqref{eq:VPE_EBC} is in $H^2(\Omega)$, then it is also a solution to \eqref{eq:VPE_original}.
\end{lemma}
\begin{proof}
Define $\phi = \nabla\cdot\vec{u}$. Then $\phi$ is a (weak) solution of the problem
\begin{equation*}
\begin{cases}[r@{~}l@{\quad}l]
\Delta\phi &= 0 &\text{in~}\Omega \\
\phi &= 0 &\text{on~}\partial\Omega\;,
\end{cases}
\end{equation*}
which has the unique solution $\phi\equiv 0$. Hence $\nabla\cdot\vec{u} = 0$ in $\Omega$.
\end{proof}
\begin{rem}
\label{rem:re-entrant_corners}
As shown in \cite{KangroNicholaides1999}, the assumption of Lemma~\ref{lem:div-free} is satisfied for the domains considered in this paper. However, it is not satisfied if the domain $\Omega$ has re-entrant (i.e., non-convex) corners. In such a case, the physically relevant (i.e., source-free) solution to \eqref{eq:VPE_original} is not in $H^1$, while problem \eqref{eq:VPE_EBC} possesses a solution in $H^1$, however, one that does not satisfy $\nabla\cdot\vec{u} = 0$ inside $\Omega$. In this paper we exclude this possibility, and for the domains considered here (see above) the problems \eqref{eq:VPE_original} and \eqref{eq:VPE_EBC} are in fact equivalent (see \cite{KangroNicholaides1999} for a proof).
\end{rem}

\subsection{Nodal Finite Elements and Babu\u{s}ka Paradox}
\label{subsec:nodal_FEM}
Among possible approaches to numerically approximate problem \eqref{eq:VPE_EBC} on an irregular domain, standard nodal-based finite elements (FE) are one of the first ideas that would come to a numerical analyst's mind. Below, we derive two possible variational formulations (\S\ref{subsubsec:FEM_variational_formulation}), and then use these to prove the possibility of the Babu\u{s}ka paradox (\S\ref{subsubsec:Babuska_paradox}). Its actual occurrence is then demonstrated via an numerical example (\S\ref{subsubsec:FEM_manufactured_solution}).

\subsubsection{Variational formulations}
\label{subsubsec:FEM_variational_formulation}
In order to conduct a FE approximation, a variational formulation of the VPE \eqref{eq:VPE_EBC} must be introduced. It is natural to work with the affine Hilbert space of vector-valued $H^1$ functions that satisfy the tangential boundary condition in \eqref{eq:VPE_EBC},
\begin{equation*}
H_{\vec{g}t}^1(\Omega)^N = \{ \vec{u} \in H^1(\Omega)^N :
\vec{n}\times(\vec{u}-\vec{g})|_{\partial\Omega} = 0\}\;.
\end{equation*}
Moreover, let $H_{0t}^1(\Omega)^N$ denote the associated homogeneous (i.e., $\vec{g} = 0$) Hilbert space. There are then two equivalent weak formulations of \eqref{eq:VPE_EBC}. To obtain the first formulation, we use the identity $-\Delta\vec{u} = \nabla\times(\nabla\times\vec{u}) - \nabla(\nabla\cdot\vec{u})$ and follow the standard procedure of multiplying the first equation in \eqref{eq:VPE_EBC} by a test function $\vec{v} \in H_{0t}(\Omega)^N$, integrating by parts, and applying the boundary conditions to the boundary integral to obtain
\begin{equation}
\label{eq:VPE_integrate_by_parts1}
\BRK{\vec{f},\vec{v}}
= \int_\Omega -\Delta\vec{u}\cdot\vec{v} \ud{x}
= a(\vec{u},\vec{v})
- \int_{\partial\Omega} (\nabla\cdot\vec{u})(\vec{n}\cdot\vec{v}) \ud{S}\;,
\end{equation}
where the bilinear form is
\begin{equation*}
a(\vec{u},\vec{v}) = \int_\Omega (\nabla\times\vec{u}) \cdot
(\nabla\times\vec{v}) + (\nabla\cdot\vec{u})(\nabla\cdot\vec{v}) \ud{x}\;.
\end{equation*}
Based on this, the first variational formulation of \eqref{eq:VPE_EBC} reads as: Given $\vec{f}\in L^2(\Omega)^N$ with $\nabla\cdot\vec{f} = 0$, find $\vec{u}\in H_{\vec{g}t}(\Omega)^N$ such that for each $\vec{v}\in H_{0t}(\Omega)^N$
\begin{equation*}
\textbf{(VP1)} \qquad\qquad
a(\vec{u},\vec{v}) = \langle \vec{f}, \vec{v} \rangle\;.\qquad\qquad\qquad
\end{equation*}
Note that due to \eqref{eq:VPE_integrate_by_parts1}, the condition $(\nabla\cdot\vec{u})_{\partial\Omega} = 0$ arises as a natural boundary condition. It is this formulation \textbf{(VP1)} that we implement in the numerical test in \S\ref{subsubsec:FEM_manufactured_solution}.

In obtaining the second variational formulation, we restrict the derivation to the case $\vec{g} = 0$, because this case is enough to show that the Babu\u{s}ka paradox can arise. Assume for a moment that $\vec{u}\in H^2(\Omega)^N$. Then, using the fact that  $\Delta\vec{u} = \nabla\cdot (\nabla\vec{u})$, we multiply the left hand side of \eqref{eq:VPE_EBC} by $\vec{v}\in H_{0t}(\Omega)^N$ and integrate by parts to obtain
\begin{equation}
\label{eq:VPE_integrate_by_parts2}
\BRK{\vec{f},\vec{v}}
= \int_\Omega -\Delta\vec{u}\cdot\vec{v} \ud{x}
= \int_\Omega \nabla\vec{u} \cdot \nabla\vec{v} \ud{x}
- \int_{\partial\Omega} \vec{v} \cdot \frac{\ud{\vec{u}}}{\ud{\vec{n}}} \ud{S}\;.
\end{equation}
Combining \eqref{eq:VPE_integrate_by_parts1} and \eqref{eq:VPE_integrate_by_parts2}, and using that $\vec{n}\times\vec{v} = 0$ on $\partial\Omega$, we can rewrite $a(\vec{u},\vec{v})$ as a new bilinear form
\begin{equation*}
b(\vec{u},\vec{v})
= \int_\Omega \nabla\vec{u} \cdot \nabla\vec{v} \ud{x}
+ \int_{\partial\Omega}
(\nabla\cdot\vec{u}-\vec{n}\cdot\frac{\ud{\vec{u}}}{\ud{\vec{n}}})
(\vec{n}\cdot\vec{v}) \ud{S}\;.
\end{equation*}
Moreover, since $\vec{n}\times\vec{u} = 0$ on $\partial\Omega$, we can expand the divergence on the boundary as
\begin{equation}
\label{eq:expansion_div_boundary}
\nabla \cdot \vec{u} = \vec{n} \cdot \frac{\ud{\vec{u}}}{\ud{\vec{n}}}
+ \kappa \vec{n} \cdot \vec{u} \qquad\text{on~}\partial\Omega\;,
\end{equation}
where $\kappa$ is the local curvature which is defined almost everywhere. Using \eqref{eq:expansion_div_boundary}, we can write $b(\vec{u},\vec{v})$ as
\begin{equation*}
b(\vec{u},\vec{v}) = \BRK{\nabla \vec{u},\nabla \vec{v}}
+ \int_{\partial\Omega} \kappa \vec{u}\cdot\vec{v} \ud{S}\;,
\end{equation*}
thus giving rise to a different variational formulation: Given $\vec{f}\in L^2(\Omega)^N$ with $\nabla\cdot\vec{f} = 0$, find $\vec{u}\in H_{0t}(\Omega)^N$ such that for each $\vec{v}\in H_{0t}(\Omega)^N$
\begin{equation*}
\textbf{(VP2)} \qquad\qquad
b(\vec{u},\vec{v}) = \langle \vec{f}, \vec{v} \rangle\;.\qquad\qquad\qquad
\end{equation*}
Clearly, by construction the bilinear forms are equal, $a(\vec{u},\vec{v}) = b(\vec{u},\vec{v})$, for functions in $H^2(\Omega)^N$. In fact, as shown in \cite{KangroNicholaides1999}, the equality also holds if the functions are in $H^1(\Omega)^N$. Moreover, the standard theory shows that the bilinear forms are coercive and continuous on $H_{0t}(\Omega)^N$ (since $\Omega$ is simply connected) so that by the Lax-Milgram theorem there is a unique solution to \textbf{(VP1)} and \textbf{(VP2)}. Moreover the variational problems \textbf{(VP1)} and \textbf{(VP2)} have the same solution.

\subsubsection{Babu\u{s}ka paradox}
\label{subsubsec:Babuska_paradox}
Using the just derived weak formulations, we prove the possible occurrence of the Babu\u{s}ka paradox. Note that other proofs have been provided before, such as in \cite{Verfurth1986}.
\begin{thm}
When solving the vector Poisson equation \eqref{eq:VPE_EBC} using a nodal FEM implementation of \textbf{(VP1)}, one may encounter the Babu\u{s}ka paradox.
\end{thm}
\begin{proof}
Suppose that $\vec{u}_h$ solves \textbf{(VP1)} using nodal elements, a triangular mesh, and a regular polygonal domain $\Omega_h$. Here $h$ denotes the diameter of the largest mesh element, so that $\Omega_h\to\Omega$ (in the appropriate sup-norm sense) as $h\to 0$. Now, under the current assumptions on the domain $\Omega$, the problems \textbf{(VP1)} and \textbf{(VP2)} have the same weak solution. Then by the equivalence of the two problems, $\vec{u}_h$ also solves \textbf{(VP2)}. For any given mesh, however, the boundary of $\Omega_h$ has flat sides with $\kappa = 0$. Consequently, the weak solution $\vec{u}_h$ solves \textbf{(VP2)} with $\kappa = 0$, i.e.,
\begin{equation*}
\int_{\Omega_h} \nabla\vec{u} \cdot \nabla\vec{v} \ud{x}
= \int_{\Omega_h} \vec{f}\cdot\vec{v} \ud{x}
\end{equation*}
for each test function $\vec{v}_h$. Therefore the solutions $\vec{u}_h$ converge to the function $\vec{u}^*$ which solves the limit problem
\begin{equation*}
\int_{\Omega} \nabla\vec{u}^* \cdot \nabla\vec{v} \ud{x}
= \langle \vec{f},\vec{v} \rangle
\end{equation*}
for each $\vec{v} \in H_{0t}(\Omega)^N$. In other words, the nodal FEM solutions $\vec{u}_h$ converge to a solution where $\kappa$ is artificially set to zero, or equivalently to a problem where one replaces the boundary condition $\nabla\cdot\vec{u} = 0$ with $\frac{\ud{\vec{u}}}{\ud{\vec{n}}} = 0$. Hence, for an arbitrary domain (with boundaries that are at least partially curved), generally $\vec{u}^*$ does not equal the true solution of the problem \textbf{(VP1)}.
\end{proof}

\begin{figure}
\begin{minipage}[b]{.49\textwidth}
\includegraphics[width=\textwidth]{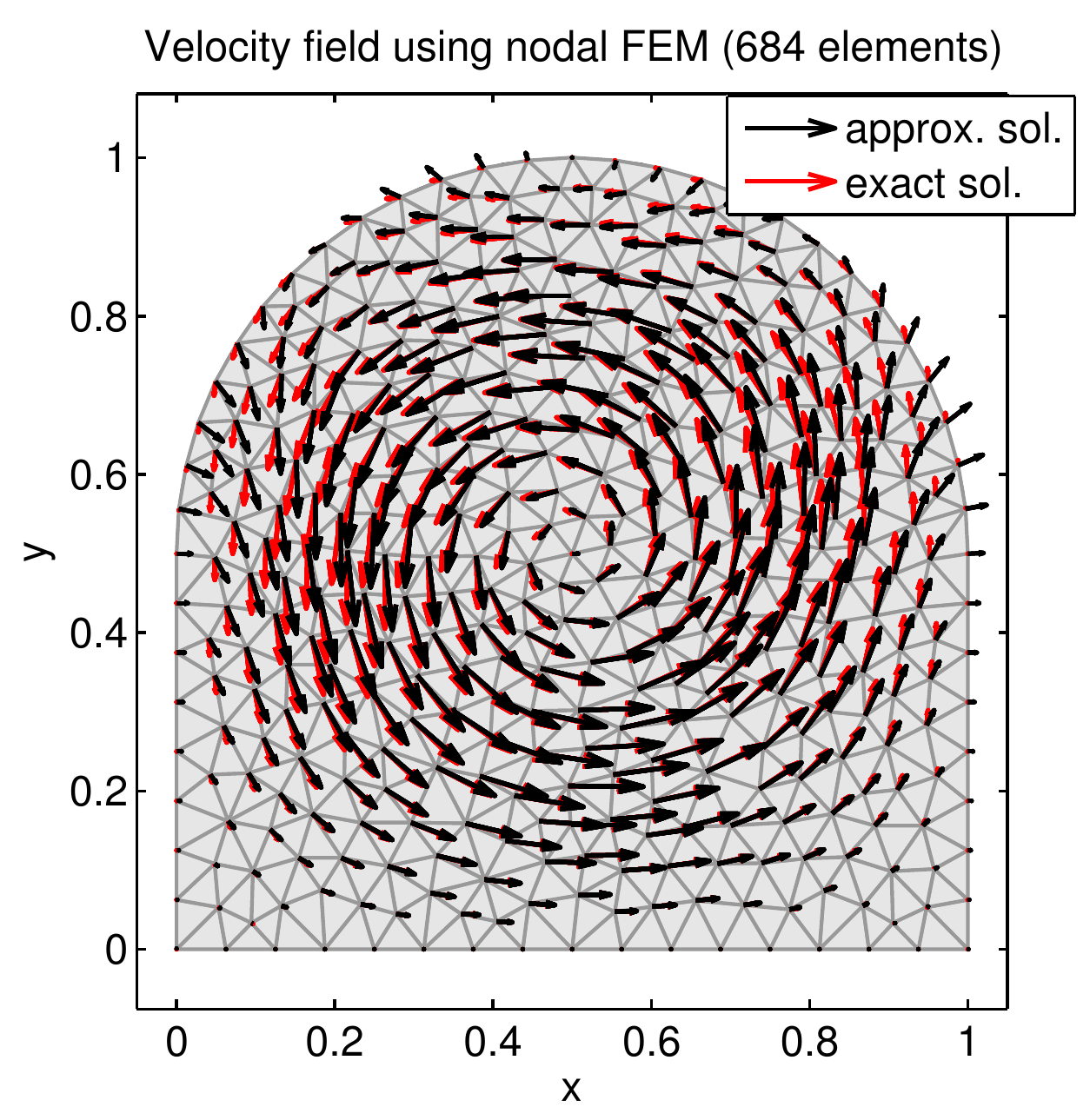}
\end{minipage}
\hfill
\begin{minipage}[b]{.49\textwidth}
\includegraphics[width=\textwidth]{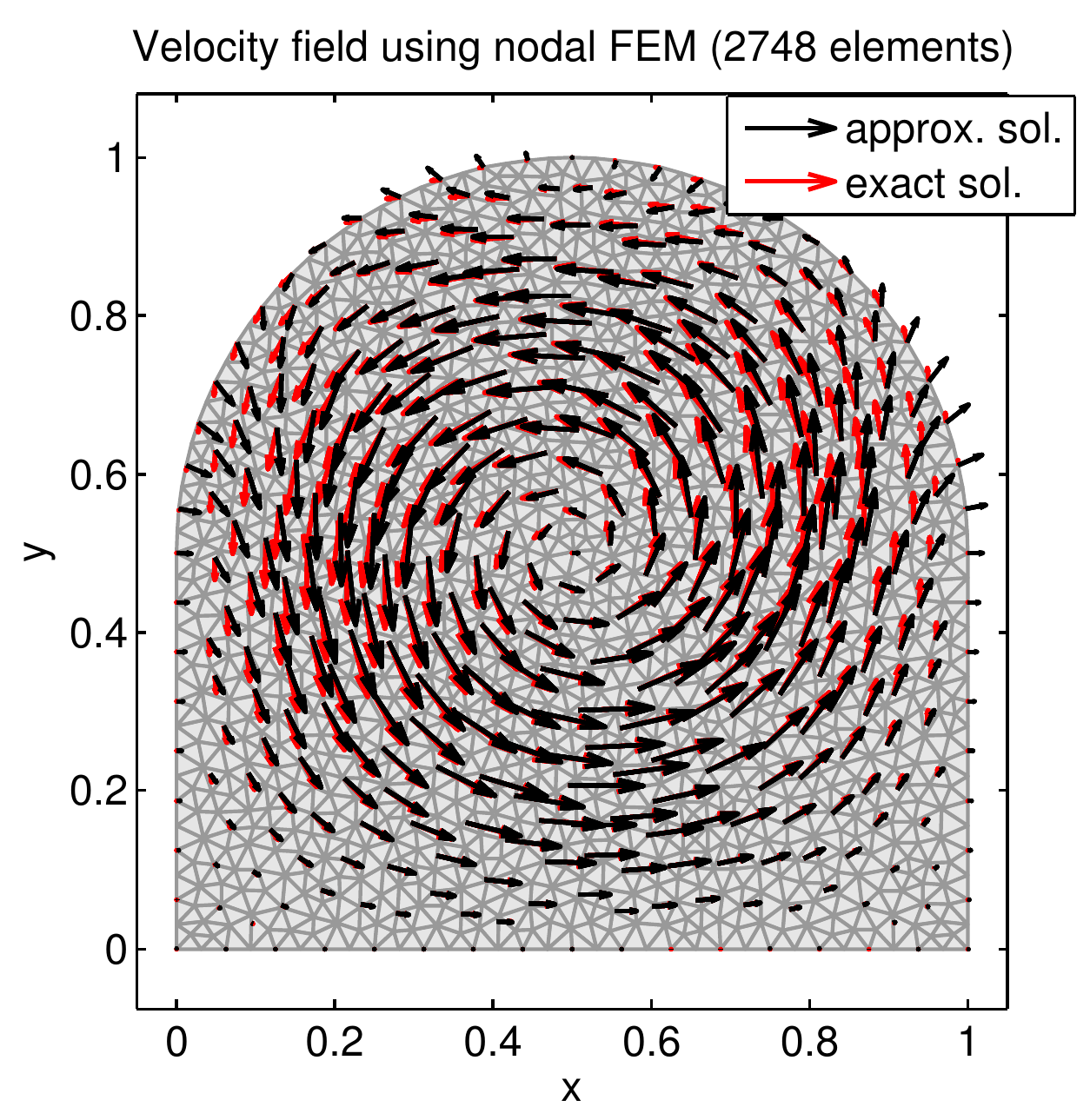}
\end{minipage}
\vspace{-.4em}
\caption{True vector field (red arrows) and numerical approximations (black arrows) obtained via nodal FEM. The numerical solutions obtained with two different mesh resolutions are almost identical, thus the numerics are essentially converged. Yet, the FEM solution differs from the true solution.}
\label{fig:FEM_solutions}
\end{figure}

\subsubsection{Manufactured solution test case}
\label{subsubsec:FEM_manufactured_solution}
We now demonstrate the Babu\u{s}ka paradox via a numerical example. On the 2D domain $\Omega=\{(x-0.5)^2+(y-0.5)^2<0.5^2\}\cup (0,1)\times (0,0.5)$, we consider the VPE \eqref{eq:VPE_EBC}. We employ the method of manufactured solutions, i.e., we prescribe the incompressible solution
\begin{equation}
\label{eq:manufactured_solution}
\vec{u}(x,y) =
\begin{pmatrix}
\phantom{-}\pi\sin(2\pi y)\sin^2(\pi x) \\
-\pi\sin(2\pi x)\sin^2(\pi y)
\end{pmatrix}\,,
\end{equation}
and set the forcing $\vec{f} = -\Delta\vec{u}$ and the boundary velocity $\vec{g} = \vec{u}$, so that the solution of the VPE \eqref{eq:VPE_EBC} recovers the prescribed solution \eqref{eq:manufactured_solution}.

We consider two triangulations (of different resolutions) that approximate the domain $\Omega$ via polygons with straight edges. To define normal vectors at the boundary vertices, we adopt the method introduced in \cite{EngelmanSani1982}, which obtains normal vectors as suitable averages of normal vectors at the edges connecting to the boundary vertex. Using standard nodal-based finite elements, or Lagrange FE \cite{Courant1943}, we implement the tangential boundary condition in an essential fashion (by choosing the solution space $H_{\vec{g}t}^1(\Omega)^N$), and leave the normal component unprescribed, with the idea that $\nabla\cdot\vec{u} = 0$ follows naturally (see above). We use quadratic $C^0$ elements.

The results of the two FEM approximations are shown in Fig.~\ref{fig:FEM_solutions}. The true vector field \eqref{eq:manufactured_solution} is shown by red arrows, and the approximate FEM solution is given by black arrows. In each case, the mesh that is used to conduct the computation is shown in the background. The right panel shows a computation with a mesh that is twice as fine as the one in the left panel. The two numerical solutions are almost identical, and hence they can be interpreted as converged (in the eye-norm) to the limit ($h\to 0$) solution $\vec{u}^*$ of the nodal FEM. Clearly the FEM solution $\vec{u}^*$ is different from the true solution $\vec{u}$, thus confirming the occurrence of the Babu\u{s}ka paradox.

\begin{rem}
Within the framework of FEM, the Babu\u{s}ka paradox can be overcome by moving to a mixed FEM formulation (cf.~\cite{RaviartThomas1977, Verfurth1986, ArnoldFalkWinther2010, ArnoldFalkGopalakrishnan2012}). The idea of mixed FEM for the VPE \eqref{eq:VPE_EBC} is introduce an additional variable, $\sigma = \nabla\times\vec{u}$, and transform $-\Delta\vec{u} = \vec{f}$ into two equations: $\sigma = \nabla\times\vec{u}$ and $\nabla\times\sigma-\nabla(\nabla\cdot\vec{u}) = \vec{f}$. This framework allows one to use Raviart-Thomas elements \cite{RaviartThomas1977} for the approximate vector field $\vec{u}^h$ (and standard nodal elements (2D) or N\'{e}d\'{e}lec elements \cite{Nedelec1980} (3D) for $\sigma^h$), and to incorporate the tangential velocity boundary condition $\vec{n}\times (\vec{u}-\vec{g}) = 0$ as boundary integrals into the weak formulation, rather than into the solution space. In a companion paper \cite{ChidyagwaiRosalesSeiboldShirokoffZhou2013}, we apply high-order mixed FEM to the pressure Poisson equation reformulations of the Navier-Stokes equations devised in \cite{ShirokoffRosales2010}. While mixed FEM overcome the Babu\u{s}ka paradox, this framework is clearly not as simple as nodal FEM, or as meshfree FD, described below.
\end{rem}

\subsection{Meshfree Finite Difference Method}
\label{subsec:VPE_meshfree_FD}
Meshfree finite differences (FD) generalize classical FD that are defined on regular grids: at a given point, a differential operator of a smooth function is approximated via a combination of function values at nearby points. The selection of points and the corresponding weights are called the \emph{stencil}. In the same way as grid-based FD, meshfree FD can be derived in two ways: as derivatives of suitable local interpolants of the data (cf.~\cite{LiszkaOrkisz1980, DuarteLiszkaTworzyako1996, Levin1998}), or via Taylor expansion of the solution (cf.~\cite{Seibold2008}). Here we outline the second methodology.

Consider a \emph{point cloud} that consists of interior points (inside $\Omega$) and boundary points (on $\partial\Omega$); see Fig.~\ref{fig:point_cloud} for an example. For a point $x_i$, let a neighborhood $B_i$ be defined. Here we employ circular neighborhoods, i.e., $B_i = \{j : \|x_j-x_i\|\le r\}$, where $r$ is an appropriately chosen radius (see below). However, many other types of neighborhoods are possible \cite{SeiboldDiss2006}. Now define the relative coordinates $\bar{x}_{ij} = x_j-x_i$ and Taylor-expand the solution $u(x)$ around $x_i$:
\begin{equation*}
u(x_j) = u(x_i)+\nabla u(x_i)\cdot\bar{x}_{ij}
+\tfrac{1}{2}\nabla^2 u(x_i):(\bar{x}_{ij}\cdot \bar{x}_{ij}^T)+\text{h.o.t.}
\end{equation*}
Note that in the quadratic term, the matrix scalar product $A:C=\sum_{i,j}A_{ij}C_{ij}$ and the outer product of $\bar{x}_{ij}$ with itself are used. While here we stop at the quadratic term, the expansion can of course be carried out further (or less far). A linear combination (with weights $a_{ij}$) of nearby solution values yields
\begin{equation}
\label{eq:stencil_linear_combination}
\sum_{j\in B_i} a_{ij} u(x_j) =
u(x_i) \sum_{j\in B_i} a_{ij}
+\nabla u(x_i) \cdot \sum_{j\in B_i} a_{ij}\bar{x}_{ij}
+\nabla^2 u(x_i) : \tfrac{1}{2}\!\sum_{j\in B_i}
a_{ij} (\bar{x}_{ij} \cdot \bar{x}_{ij}^T)
+\text{h.o.t.}
\end{equation}
If \eqref{eq:stencil_linear_combination} is supposed to approximate a given differential operator applied to the solution, then the stencil weights $a_{ij}$ must satisfy certain constraints. For instance, for \eqref{eq:stencil_linear_combination} to approximate $\Delta u(x_i)$, it is required that
\begin{equation*}
\sum_{j\in B_i} a_{ij} = 0 \; , \quad
\sum_{j\in B_i} \bar{x}_{ij} a_{ij} = 0 \; \text{, and} \quad
\sum_{j\in B_i} (\bar{x}_{ij}\cdot\bar{x}_{ij}^T) a_{ij} = 2I\;,
\end{equation*}
which in 2D gives rise to the linear system of constraints
\begin{equation}
\label{eq:constraints_Laplacian_first_order}
\underbrace{\begin{pmatrix}
\bar{x}_{i,j_1} & \dots & \bar{x}_{i,j_{m_i}} \\
\bar{y}_{i,j_1} & \dots & \bar{y}_{i,j_{m_i}} \\
\bar{x}_{i,j_1}^2 & \dots & \bar{x}_{i,j_{m_i}}^2 \\
\bar{x}_{i,j_1}\bar{y}_{i,j_1} & \dots & \bar{x}_{i,j_{m_i}}\bar{y}_{i,j_{m_i}} \\
\bar{y}_{i,j_1}^2 & \dots & \bar{y}_{i,j_{m_i}}^2
\end{pmatrix}}_{V_i}
\cdot
\underbrace{\begin{pmatrix} a_{i,j_1} \\ \vdots \\ \vdots \\
a_{i,j_{m_i}} \end{pmatrix}}_{\vec{a}_i}
=
\underbrace{\begin{pmatrix} 0 \\ 0 \\ 2 \\ 0 \\ 2 \end{pmatrix}}_{\vec{b}}\,.
\end{equation}
Here, the stencil vector $\vec{a}_i$ does not contain the diagonal entry $a_{ii}$. Its value is obtained as $a_{ii} = -\sum_{j\in B_i\setminus\{i\}} a_{ij}$. Moreover, $m_i = |B_i|-1$ is the number of neighbors of $x_i$. If the radius $r$ is chosen large enough that $m_i\ge 5\;\forall\,i$, and if the point cloud generation (see \S\ref{subsec:point_cloud_generation}) ensures that no pathological point configurations arise (see \cite{SeiboldDiss2006} for examples), then system \eqref{eq:constraints_Laplacian_first_order} always has a solution, and the resulting approximation is (at least) first order accurate.

If $m_i>5$, system \eqref{eq:constraints_Laplacian_first_order} in general has infinitely many solutions. One way (employed here) to single out a unique solution is via a weighted least-squares (WLSQ) minimization problem
\begin{equation}
\label{eq:stencil_WLSQ}
\min\!\sum_{j\in B_i\setminus\{i\}}\!\frac{a_{ij}^2}{w_{ij}}\;,\
\text{s.t.}\ V_i\cdot\vec{a}_i = \vec{b}
\end{equation}
where the weights are decreasing with the distance, $w_{ij} = \|x_j-x_i\|_2^{-\beta}$ (here we choose $\beta = 2$). The solution of \eqref{eq:stencil_WLSQ} is
\begin{equation*}
\vec{a}_i = W_iV_i^T(V_iW_iV_i^T)^{-1}\cdot\vec{b}\;.
\end{equation*}
where $W = \text{diag}(w_{i,1},\dots,w_{i,j_{m_i}})$. Note that an alternative approach (not employed here) would be to solve system \eqref{eq:constraints_Laplacian_first_order} in an $\ell^1$ sense, i.e.,
\begin{equation*}
\min\!\sum_{j\in B_i\setminus\{i\}}\!\frac{a_{ij}}{w_{ij}}\;,\
\text{s.t.}\ V_i\cdot\vec{a}_i = \vec{b}\;,\ \vec{a}_i\ge 0\;,
\end{equation*}
which would generate optimally sparse stencils \cite{Seibold2008}.

Other differential operators are approximated in an analogous fashion. For instance, a first order approximation to $\partial_x u$ is obtained by setting
\begin{equation*}
V_i =
\begin{pmatrix}
\bar{x}_{i,j_1} & \dots & \bar{x}_{i,j_{m_i}} \\
\bar{y}_{i,j_1} & \dots & \bar{y}_{i,j_{m_i}}
\end{pmatrix}
\,,\
\vec{b} = \begin{pmatrix} 1 \\ 0 \end{pmatrix}\,,
\end{equation*}
and a second order approximation to $\partial_x u$ is obtained by setting
\begin{equation*}
V_i =
\begin{pmatrix}
\bar{x}_{i,j_1} & \dots & \bar{x}_{i,j_{m_i}} \\
\bar{y}_{i,j_1} & \dots & \bar{y}_{i,j_{m_i}} \\
\bar{x}_{i,j_1}^2 & \dots & \bar{x}_{i,j_{m_i}}^2 \\
\bar{x}_{i,j_1}\bar{y}_{i,j_1} & \dots & \bar{x}_{i,j_{m_i}}\bar{y}_{i,j_{m_i}} \\
\bar{y}_{i,j_1}^2 & \dots & \bar{y}_{i,j_{m_i}}^2
\end{pmatrix}
\,,\
\vec{b} = \begin{pmatrix} 1 \\ 0 \\ 0 \\ 0 \\ 0 \end{pmatrix}\,.
\end{equation*}

Applying this procedure to each equation (at all interior points) and boundary condition (at all boundary points) of the vector Poisson equation \eqref{eq:VPE_EBC} leads (here is 2D) to the linear system
\begin{equation}
\label{eq:VPE_linear_system}
\begin{footnotesize}
\begin{array}{|@{~}c@{~}c@{~}c@{~}c@{~}c@{~}c@{~}|
@{~}c@{~}c@{~}c@{~}c@{~}c@{~}c@{~}|@{~}c@{~}c@{~}c@{~}c@{~}|@{~}c@{~}c@{~}c@{~}c@{~}|}
\hline
*& &*&*& &*  &   & & & & &   &   &*&*&   &   & & &  \\[-.4em]
 &*&*& &*&   &   & & & & &   &   & & &*  &   & & &  \\[-.4em]
*&*&*& & &*  &   & & & & &   &  *& &*&*  &   & & &  \\[-.4em]
*& & &*&*&   &   & & & & &   &  *& &*&   &   & & &  \\[-.4em]
 &*& &*&*&*  &   & & & & &   &  *&*& &   &   & & &  \\[-.4em]
*& &*& &*&*  &   & & & & &   &  *& & &*  &   & & &  \\
\hline
 & & & & &   &  *& &*&*& &*  &   & & &   &   &*&*&  \\[-.4em]
 & & & & &   &   &*&*& &*&   &   & & &   &   & & &* \\[-.4em]
 & & & & &   &  *&*&*& & &*  &   & & &   &  *& &*&* \\[-.4em]
 & & & & &   &  *& & &*&*&   &   & & &   &  *& &*&  \\[-.4em]
 & & & & &   &   &*& &*&*&*  &   & & &   &  *&*& &  \\[-.4em]
 & & & & &   &  *& &*& &*&*  &   & & &   &  *& & &* \\
\hline
 & &*& &*&*  &   & &*& &*&*  &  *& &*&   &  *& &*&  \\[-.4em]
*&*& & &*&   &  *&*& & &*&   &   &*& &*  &   &*& &* \\[-.4em]
 &*&*& & &*  &   &*&*& & &*  &  *& &*&*  &  *& &*&* \\[-.4em]
*& & &*& &   &  *& & &*& &   &   &*&*&*  &   &*&*&* \\
\hline
 & & & & &   &   & & & & &   &  *& & &   &  *& & &  \\[-.4em]
 & & & & &   &   & & & & &   &   &*& &   &   &*& &  \\[-.4em]
 & & & & &   &   & & & & &   &   & &*&   &   & &*&  \\[-.4em]
 & & & & &   &   & & & & &   &   & & &*  &   & & &* \\
\hline
\end{array}
\cdot
\begin{array}{|@{~}c@{~}|}
\hline
u^x_1 \\[-.3em] \cdot \\[-.6em] \cdot \\[-.6em] \cdot \\[-.6em]
\cdot \\[-.1em] u^x_{N_\text{i}} \\[.2em]
\hline
u^y_1 \\[-.3em] \cdot \\[-.6em] \cdot \\[-.6em] \cdot \\[-.6em]
\cdot \\[-.1em] u^y_{N_\text{i}} \\[.2em]
\hline
u^x_{N_\text{i}+1} \\[-.3em] \cdot \\[-.6em] \cdot \\[-.4em] u^x_N \\[.1em]
\hline
u^y_{N_\text{i}+1} \\[-.3em] \cdot \\[-.6em] \cdot \\[-.4em] u^y_N \\[.1em]
\hline
\end{array}
=
\begin{array}{|@{}c@{}|}
\hline
f^x_1 \\[-.2em] \cdot \\[-.6em] \cdot \\[-.6em] \cdot \\[-.6em]
\cdot \\[-.2em] f^x_{N_\text{i}} \\[.2em]
\hline
f^y_1 \\[-.2em] \cdot \\[-.6em] \cdot \\[-.6em] \cdot \\[-.6em]
\cdot \\[-.2em] f^y_{N_\text{i}} \\[.2em]
\hline
0 \\[-.5em] \cdot \\[-.6em] \cdot \\[-.2em] 0 \\[.1em]
\hline
g_{N_\text{i}+1} \\[-.2em] \cdot \\[-.6em] \cdot \\[-.4em]
g_N \\[.1em]
\hline
\end{array}
\;
\begin{array}{l}
\left.\rule{0pt}{2.80em}\right\}\;
\Delta u^x = f^x \quad\text{in~}\Omega \\
~\\[-1em]
\left.\rule{0pt}{2.80em}\right\}\;
\Delta u^y = f^y \quad\text{in~}\Omega \\
~\\[-1em]
\left.\rule{0pt}{1.90em}\right\}\;
\partial_x u^x+\partial_y u^y = 0 \quad\text{on~}\partial\Omega \\
~\\[-1em]
\left.\rule{0pt}{1.90em}\right\}\;
n^x u^y-n^y u^x = g \quad\text{on~}\partial\Omega
\end{array}
\end{footnotesize}
\end{equation}
In this system, $\vec{u} = (u^x,u^y)$, $\vec{f} = (f^x,f^y)$, and the function $g = \vec{n}\times\vec{g}$ at the boundary, and $g_i = \vec{n}_i\times\vec{g}_i$. Moreover, the total number of points is $N$, and the number of interior points is $N_\text{i}$. In the sparse block matrix, the first two rows of blocks correspond to the two components of the Poisson equation at the interior points; the third block row encodes the $\nabla\cdot\vec{u} = 0$ boundary condition; and the fourth block row represents the $\vec{n}\times (\vec{u}-\vec{g}) = 0$ condition. The first two block columns corresponds to the two vector field components at the interior points $x_1,\dots,x_{N_\text{i}}$; and likewise the last two block columns correspond to $\vec{u} = (u^x,u^y)$ at the boundary points $x_{N_\text{i}},\dots,x_N$. Each empty block is devoid of nonzero entries, and each block with stars possesses some nonzero entries, corresponding to the neighboring points of each central point.

\subsection{Point Cloud Generation}
\label{subsec:point_cloud_generation}
The generation of point clouds is conducted similarly to the mesh generation code \textsf{DistMesh} \cite{PerssonStrang2004}. The domain $\Omega$ is specified via a level set function $\phi$, chosen so that $\Omega = \{x : \phi(x)<0\}$. Hence, $\phi$ also defines the boundary $\partial\Omega = \{x : \phi(x)=0\}$ and surface normal vectors $\vec{n} = \nabla\phi/|\nabla\phi|$. The access to the level set function allows an immediate check whether a given location/point is inside or outside of the domain. Moreover, if $\phi$ is a signed distance function (i.e., $|\nabla\phi| = 1$ a.e.), one even has immediate access to a location's distance to the domain boundary.

\begin{figure}
\centering
\begin{minipage}[b]{.74\textwidth}
\centering
\includegraphics[width=.64\textwidth]{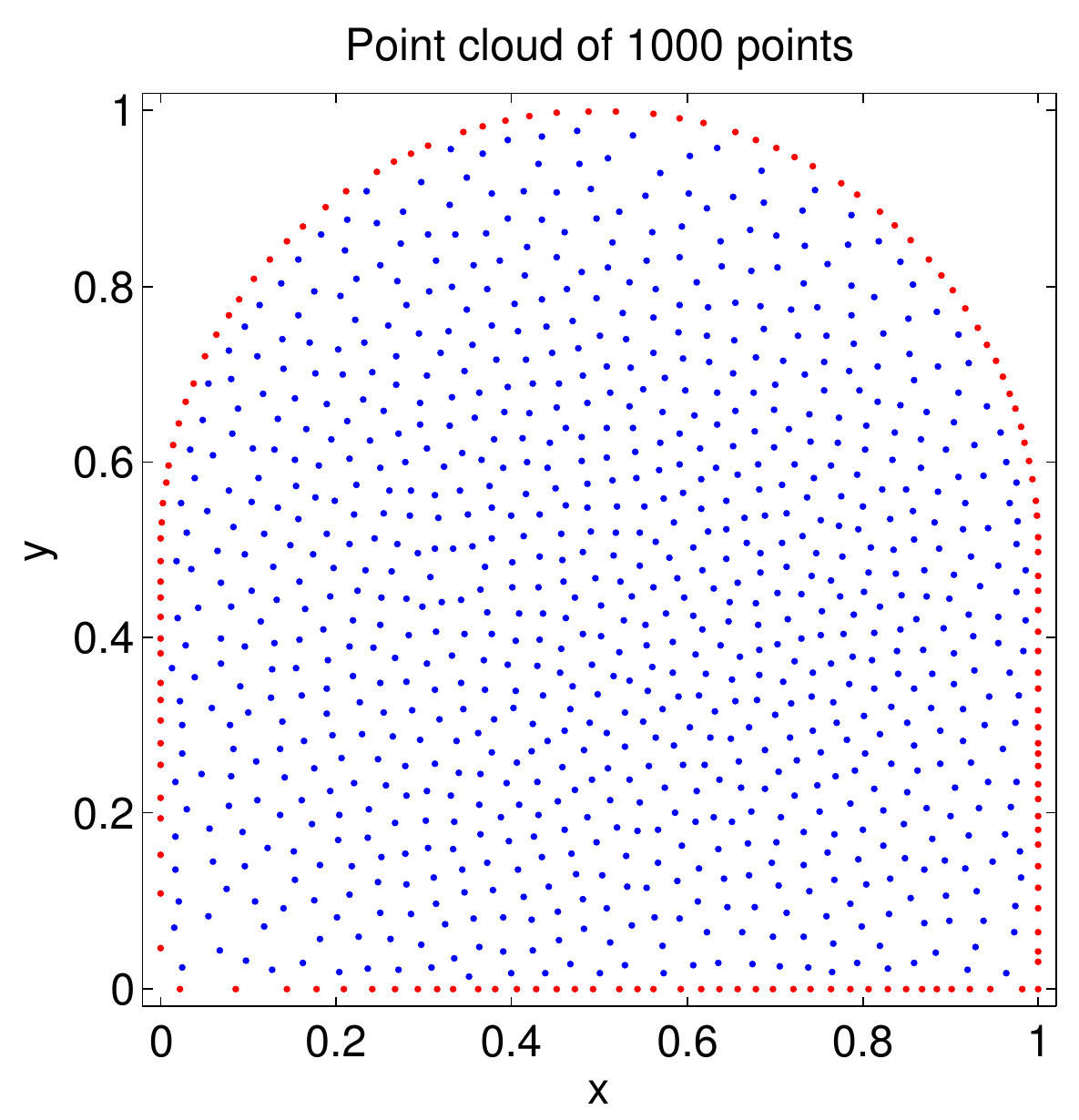}
\vspace{-.4em}
\caption{Point cloud for the computational domain $\Omega$ with 1000 points. The boundary points are shown in red and the interior points in blue.}
\label{fig:point_cloud}
\end{minipage}
\end{figure}

The generation of a cloud of $N$ points is initiated by placing $N$ points randomly inside $\Omega$. After that, the points are moved according to repulsive forces experienced from nearby points, i.e., the point $x_i$ moves according to
\begin{equation}
\label{eq:point_motion}
\dot{x}_i = \!\sum_{j\in B_i\setminus\{i\}}\!
\min\{\|x_j-x_i\|_2^{-2},v_\text{max}\}\;,
\end{equation}
where $B_i$ is a circular neighborhood around $x_i$, and $v_\text{max}$ is some upper bound on the repulsion. The law of motion \eqref{eq:point_motion} is further constrained by $x_i\in\bar{\Omega}$, i.e., no point can ever move beyond the domain boundary to leave the domain. A simple way to implement this constraint is to allow points to slightly leave $\bar{\Omega}$, but then to immediately project them back onto $\partial\Omega$, using the normal $\nabla\phi/|\nabla\phi|$ that is defined also outside of $\bar{\Omega}$. Finally, points that are inside $\Omega$ but too close to the boundary are also projected onto $\partial\Omega$, thus preventing interior points from being too close to the boundary (see \cite{Seibold2008} for why this would be undesirable). The law of motion \eqref{eq:point_motion} is then applied to all points until the amount of motion has fallen below a given threshold. The resulting point clouds are unstructured, and tend to be quite uniform (i.e., the ratio between the minimum distance between points and the radius of the largest ball that contains no points (cf.~\cite{Levin1998}) is quite large). An example of a point cloud associated with the domain defined in \S\ref{subsubsec:FEM_manufactured_solution} is shown in Fig.~\ref{fig:point_cloud}.

For the definition of a ``mesh'' resolution $h$ of a given point cloud, there is a variety of possible choices \cite{Levin1998}. Here we use a simple averaged concept of resolution, defined as follows. One type of configuration in which nearby points are equidistant is an optimal sphere packing, which in 2D is a hexagonal lattice, i.e., adjacent points form equilateral triangles. Each point is a corner of six triangles, and each triangle is shared by three points. One can therefore associate to each point 1/3 of each of the six triangles, resulting in an area per point of $A = \frac{\sqrt{3}}{2}h^2$, if the spacing between points is $h$. Counting the area of the $N_\text{i}$ interior points full, and of the $N_\text{b}$ boundary points half, and equating the total ``point area'' with the area of the domain, $\lambda(\Omega)$, we obtain the expression $h = \sqrt{\frac{4\lambda(\Omega)}{\sqrt{3}(2N_\text{i}+N_\text{b})}}$ for the resolution.

\begin{figure}
\begin{minipage}[b]{.32\textwidth}
\includegraphics[width=\textwidth]{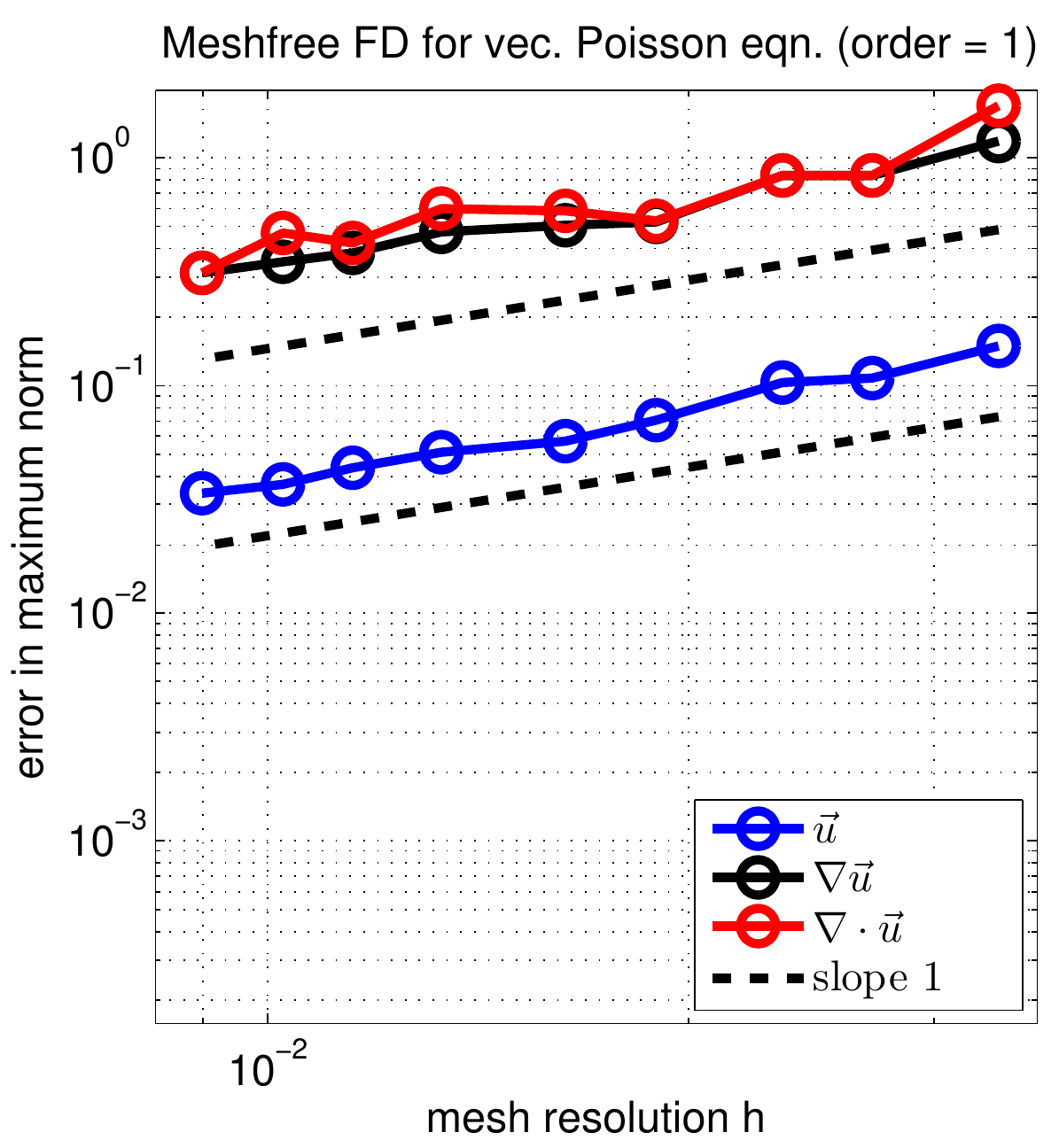}
\end{minipage}
\hfill
\begin{minipage}[b]{.32\textwidth}
\includegraphics[width=\textwidth]{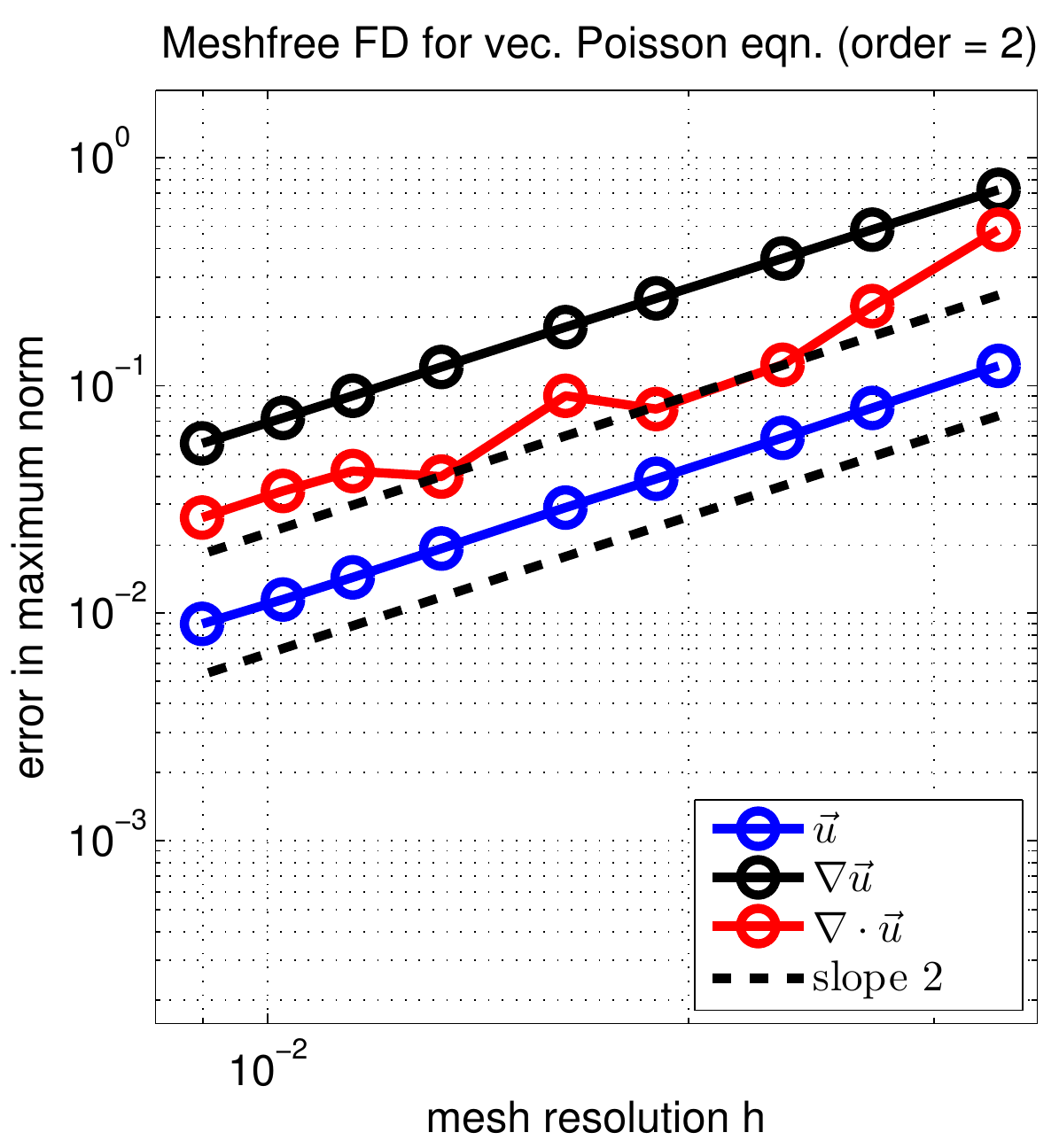}
\end{minipage}
\hfill
\begin{minipage}[b]{.32\textwidth}
\includegraphics[width=\textwidth]{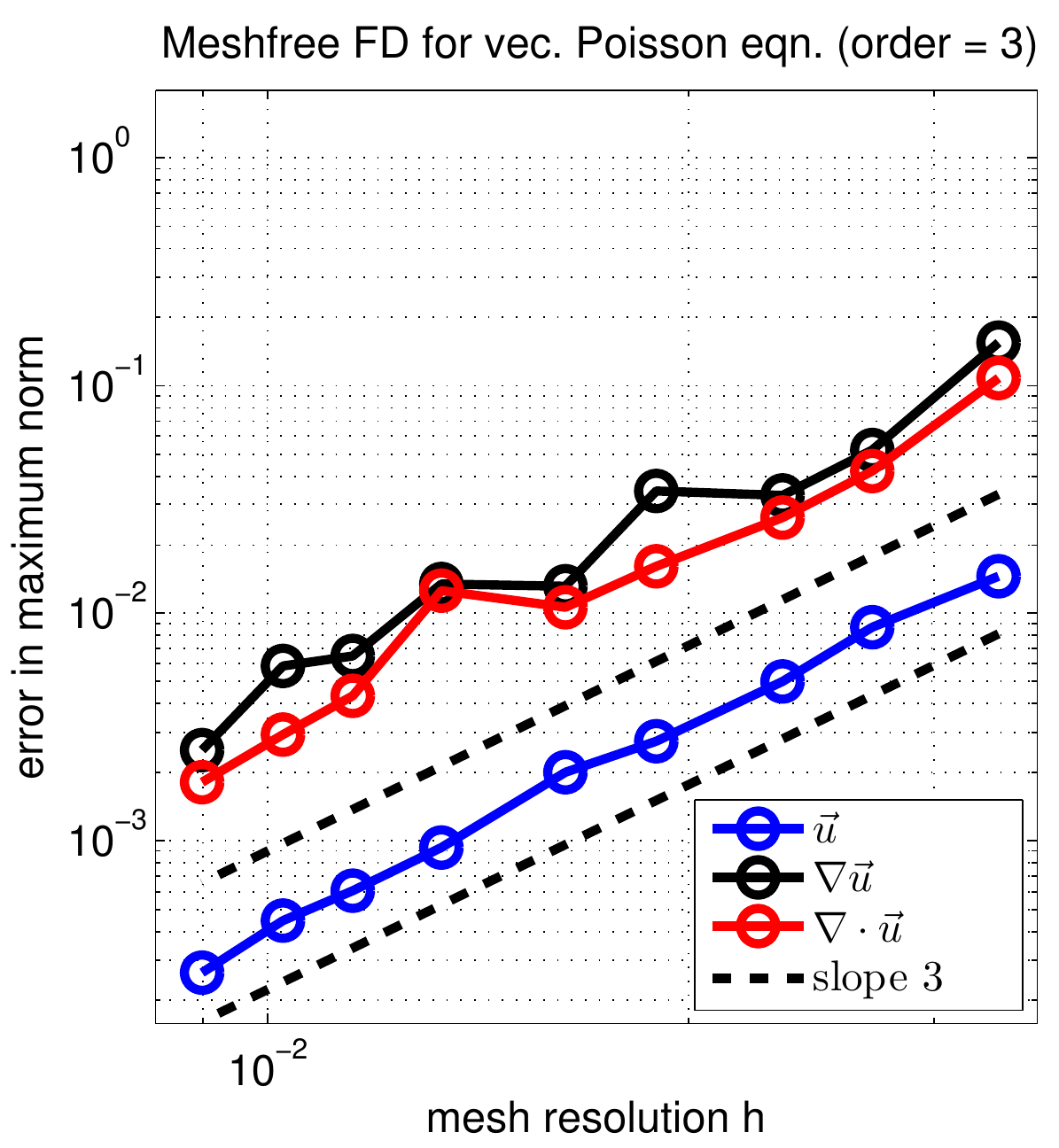}
\end{minipage}
\vspace{-.4em}
\caption{Error convergence for first (left), second (middle), and third (right) order meshfree FD approximations. The errors are measured in the maximum norm. The results show that $k^\text{th}$ order meshfree FD stencils result in $k^\text{th}$ order convergence rates for the solution of the VPE, and its derivatives.}
\label{fig:VPE_meshfree_solutions}
\end{figure}

\subsection{Numerical Results}
\label{subsec:VPE_numerical_results}
We consider the same manufactured solution test problem as studied in \S\ref{subsubsec:FEM_manufactured_solution}. On the domain $\Omega=\{(x-0.5)^2+(y-0.5)^2<0.5^2\}\cup (0,1)\times (0,0.5)$ (see Figs.~\ref{fig:FEM_solutions} and \ref{fig:point_cloud}), point clouds of various numbers of points are generated, so that convergence studies can be conducted. For each point cloud, the vector Poisson equation \eqref{eq:VPE_EBC} is discretized via the procedure described in \S\ref{subsec:VPE_meshfree_FD}. We conduct the meshfree FD approximation for three different orders: first order accuracy, i.e., the Taylor expansion in \eqref{eq:stencil_linear_combination} is carried out up to the quadratic term for $\Delta\vec{u}$ and up to the linear term for $\nabla\cdot\vec{u}$; and second and third order accuracies, for which the Taylor expansions in \eqref{eq:stencil_linear_combination} are carried out further accordingly.

The numerical approximations obtained for the different orders and mesh resolutions are then compared to the true solution in the maximum norm, taken over all points. We consider the errors in the vector field $\vec{u}$ itself, as well as its Jacobian $\nabla\vec{u}$ (which is important for calculating stresses at the boundary when $\vec{u}$ represents a velocity field) and its divergence $\nabla\cdot\vec{u}$ (which by the equivalence of problems \eqref{eq:VPE_original} and \eqref{eq:VPE_EBC} should be close to zero). All derivative quantities are obtained from the vector field $\vec{u}$ via meshfree FD stencils of fourth order. Hence, if an order of less than 4 is observed, we know that this is the true accuracy of the numerical result.

The error convergence of these quantities is shown in Fig.~\ref{fig:VPE_meshfree_solutions}, for the approximation orders 1 (left panel), 2 (middle panel), and 3 (right panel). The results show that all approaches converge as $h\to 0$, and the convergence orders equal the local approximation orders. In particular, the convergence orders of the derivative quantities are the same as those of the vector field itself. This is an important property that finite difference methods commonly exhibit, and that is in contrast to finite element methods that frequently lose an order of accuracy when a derivative quantity is evaluated (in the sense of functions).

\vspace{1.5em}
\section{Pressure Poisson Equation Reformulation of the Navier-Stokes Equations}
\label{sec:PPE_reformulation}
We consider the time-dependent incompressible Navier-Stokes equations (NSE)
\begin{equation}
\label{eq:NSE}
\begin{cases}[r@{~}l@{\quad}l]
\partial_t\vec{u} + (\vec{u}\cdot\nabla)\vec{u}
&= -\nabla p + \nu\Delta\vec{u} + \vec{f} &\text{in~}\Omega\times (0,T) \\
\vec{u} &= \vec{g} &\text{on~}\partial\Omega\times (0,T) \\
\vec{u} &= \mathring{\vec{u}} &\text{on~}\Omega\times\{t=0\}\;,
\end{cases}
\end{equation}
with compatibility conditions
\begin{align*}
\mathring{\vec{u}} &= \vec{g}\phantom{0} \ \ \text{on~}\partial\Omega\times \{t=0\}
& \text{(continuity between i.c. and b.c.)} \\
\nabla\cdot\mathring{\vec{u}} &= 0\phantom{\vec{g}} \ \ \text{in~}\Omega
& \text{(incompressible i.c.)} \\
\int_{\partial\Omega} \vec{n}\cdot\vec{g} \ud{x} &= 0\;.
& \text{(inflow = outflow)}
\end{align*}
Due to the lack of a time evolution of the pressure, there is no single canonical way to numerically advance \eqref{eq:NSE} forward in time. One class of approaches to do so is based on approximating the time derivative in the momentum equation, and solving for $\vec{u}$ and $p$ in a fully coupled fashion. This methodology is accurate, but also costly, because a large system must be solved that possesses a saddle point structure. An alternative class of approaches decouples the pressure solve from the velocity update. This methodology was first proposed in the form of projection methods \cite{Chorin1968, Temam1969}, and later employed in approaches based on pressure Poisson equation (PPE) reformulations of the NSE \cite{Henshaw1994, HenshawKreissReyna1994, JohnstonLiu2002, HenshawPetersson2003, ShirokoffRosales2010}.

A fundamental difference between projection methods (see \cite{GuermondMinevShen2006} for an overview) and PPE reformulations is that projection methods are based on a fractional step approach in which the time-evolution of the velocity field and its projection onto the space of incompressible fields are alternated. In contrast, PPE reformulations derive an equation for the pressure that replaces the incompressibility constraint in \eqref{eq:NSE} by a global pressure function $p = P(\vec{u})$ that is designed so that the solutions of the PPE reformulation are identical to the solutions of the original NSE. As a consequence, numerical methods based on PPE reformulations are structurally easy to extend to high order accuracy in time. In addition, they do not suffer from poor spatial accuracy near boundaries. A difficulty of PPE reformulations is that the Poisson equation for the pressure can involve complicated expressions, whose interaction with the velocity field equation is not always easy to understand and analyze. Another important property of PPE reformulations is that, unlike the original NSE, they are also defined if the initial conditions are not incompressible, see \S\ref{subsec:PPE_EBC}.

\subsection{PPE Reformulation with Electric Boundary Conditions}
\label{subsec:PPE_EBC}
In this paper, we are concerned with the particular PPE reformulation of the NSE proposed in \cite{ShirokoffRosales2010}. Its fundamental difference from previously proposed PPE reformulations \cite{Henshaw1994, HenshawKreissReyna1994, JohnstonLiu2002} is that the velocity field satisfies electric boundary conditions, i.e., incompressibility and the tangential flow are prescribed at the boundary. In turn, the normal velocity is enforced via a relaxation term in the pressure equation (see \cite{ShirokoffRosales2010} for a discussion on the choice of $\lambda$). The PPE reformulation consists of the momentum equation
\begin{equation}
\label{eq:NSE_EBC}
\begin{cases}[r@{~}l@{\quad}l]
\partial_t\vec{u} + (\vec{u}\cdot\nabla)\vec{u}
&= -\nabla P(\vec{u}) + \nu\Delta\vec{u} + \vec{f} &\text{in~}\Omega\times (0,T) \\
\nabla\cdot\vec{u} &= 0 &\text{on~}\partial\Omega\times (0,T) \\
\vec{n}\times\vec{u} &= \vec{n}\times\vec{g} &\text{on~}\partial\Omega\times (0,T) \\
\vec{u} &= \mathring{\vec{u}} &\text{on~}\Omega\times\{t=0\}\;,
\end{cases}
\end{equation}
where $P(\vec{u})$ is a solution of the associated pressure Poisson equation
\begin{equation}
\label{eq:PPE_EBC}
\begin{cases}[r@{~}l@{\quad}l]
\Delta p &= \nabla\cdot (\vec{f} - (\vec{u}\cdot\nabla)\vec{u})
&\text{in~}\Omega \\
\frac{\partial p}{\partial \vec{n}}
&= \vec{n}\cdot (\vec{f} - \partial_t\vec{g}
+ \nu\Delta\vec{u} - (\vec{u}\cdot\nabla)\vec{u})
+ \lambda\vec{n}\cdot (\vec{u}-\vec{g}) &\text{on~}\partial\Omega\;.
\end{cases}
\end{equation}
While a variety of modifications and additions can be applied to these equations (cf.~\cite{ShirokoffRosales2010, ChidyagwaiRosalesSeiboldShirokoffZhou2013}), here we study the equations exactly in the given form (with one small caveat regarding the compatibility of the pressure boundary conditions, see below).

If the initial conditions are not incompressible, i.e., $\nabla\cdot\mathring{\vec{u}} \neq 0$, then the solution of \eqref{eq:NSE_EBC} relaxes towards a solutions of \eqref{eq:NSE}, for the following reason. Let $\phi = \nabla\cdot\vec{u}$. Then the application of $\nabla\cdot$ to the momentum equation in \eqref{eq:NSE}, and the use of the first equation in \eqref{eq:PPE_EBC} yields that $\phi$ satisfies the heat equation with homogeneous Dirichlet boundary conditions
\begin{equation}
\label{eq:PPE_heat_equation_div}
\begin{cases}[r@{~}l@{\quad}l]
\partial_t\phi &= \nu\Delta\phi &\text{in~}\Omega\times (0,T) \\
\phi &= 0 &\text{on~}\partial\Omega\times (0,T) \\
\phi &= \nabla\cdot\mathring{\vec{u}} &\text{on~}\Omega\times\{t=0\}\;.
\end{cases}
\end{equation}
This property is of great relevance. It means that in PPE reformulations, there is no need to impose a discrete incompressibility principle. If, due to numerical approximation errors, the numerical solution starts to drift away from the $\nabla\cdot\vec{u} = 0$ manifold, equation \eqref{eq:PPE_heat_equation_div} ensures that it is pulled back towards incompressibility.

At the same time, the fact that the numerical solution may not be exactly incompressible, implies that the compatibility condition in the pressure Poisson equation \eqref{eq:PPE_EBC} may be violated. Specifically, \eqref{eq:PPE_EBC} has a solution if
\begin{equation}
\label{eq:PPE_compatibility_condition}
\int_\Omega (\lambda+\nu\Delta)\phi\ud{x}
- \int_{\partial\Omega} (\partial_t+\lambda)\vec{g}\cdot\vec{n}\ud{S} = 0\;,
\end{equation}
and due to numerical approximation errors (or because a problem with $\nabla\cdot\mathring{\vec{u}}\neq 0$ is considered) this condition may be violated. However, whenever this occurs, the solution of the augmented system \eqref{eq:augmented_pressure_system}, described in \S\ref{subsec:PPE_meshfree_FD} projects the right hand side of \eqref{eq:PPE_EBC} in a way that the solvability condition is satisfied.

Below, we first generalize the meshfree finite difference methods developed in \S\ref{sec:vector_poisson_equation} for the vector Poisson equation to the vector heat equation (\S\ref{subsec:VHE_meshfree_FD}). Then, we extend the methodology to the PPE reformulation (\S\ref{subsec:PPE_meshfree_FD}).

\subsection{Meshfree Finite Differences for the Vector Heat Equation}
\label{subsec:VHE_meshfree_FD}
Before moving to the full PPE reformulation \eqref{eq:NSE_EBC}, we first generalize the numerical scheme developed in \S\ref{subsec:VPE_meshfree_FD} for the vector Poisson equation \eqref{eq:VPE_EBC} to the vector heat equation (VHE) that describes the evolution of a vector field $\vec{u}(x,t)$ via the system
\begin{equation}
\label{eq:VHE_EBC}
\begin{cases}[r@{~}l@{\quad}l]
\partial_t\vec{u} &= \nu\Delta\vec{u}+\vec{f} &\text{in~}\Omega\times (0,T) \\
\nabla\cdot\vec{u} &= 0 &\text{on~}\partial\Omega\times (0,T) \\
\vec{n}\times\vec{u} &= \vec{n}\times\vec{g} &\text{on~}\partial\Omega\times (0,T) \\
\vec{u} &= \mathring{\vec{u}} &\text{on~}\Omega\times\{t=0\}\;.
\end{cases}
\end{equation}
Here the forcing $\vec{f}(x,t)$ and the initial data $\mathring{\vec{u}}(x)$ are incompressible, i.e., $\nabla\cdot\vec{f} = 0$ and $\nabla\cdot\mathring{\vec{u}} = 0$. We discretize \eqref{eq:VHE_EBC} in time via ImEx (Implicit-Explicit) schemes \cite{AscherRuuthSpiteri1997}. Specifically, we use the first-order scheme
\begin{equation}
\label{eq:IMEX_first}
\tfrac{1}{\Delta t}\prn{\vec{u}(t+\Delta t)-\vec{u}(t)}
= \vec{R}(\vec{u}(t+\Delta t))+\vec{Q}(\vec{u}(t))
\end{equation}
and the second-order two-stage Runge-Kutta scheme
\begin{equation}
\label{eq:IMEX_second}
\begin{split}
\tfrac{1}{\Delta t}\prn{\vec{u}^*-\vec{u}(t)}
&= \gamma\vec{R}(\vec{u}^*)+\gamma\vec{Q}(\vec{u}(t)) \\
\tfrac{1}{\Delta t}\prn{\vec{u}(t+\Delta t)-\vec{u}(t)}
&= \gamma\vec{R}(\vec{u}(t+\Delta t))+(1-\gamma)\vec{R}(\vec{u}^*)
+\delta\vec{Q}(\vec{u}(t))+(1-\delta)\vec{Q}(\vec{u}^*)\;,
\end{split}
\end{equation}
where $\gamma = 1-\frac{1}{2}\sqrt{2}$ and $\delta = 1-\frac{1}{2\gamma}$. For the VHE \eqref{eq:VHE_EBC}, an explicit first-order scheme (forward Euler) is obtained by setting $\vec{R}(\vec{u}) = 0$ and $\vec{Q}(\vec{u}) = \nu\Delta\vec{u}+\vec{f}$ in \eqref{eq:IMEX_first}. In turn, setting $\vec{R}(\vec{u}) = \nu\Delta\vec{u}$ and $\vec{Q}(\vec{u}) = \vec{f}$ yields semi-implicit schemes of first order via \eqref{eq:IMEX_first}, and of second order via \eqref{eq:IMEX_second}. Moreover, the boundary conditions are always treated implicitly, so that they are satisfied by the new state (at the end of each Runge-Kutta stage).

The schemes that treat $\Delta\vec{u}$ implicitly lead to relatively simple modifications of the linear system \eqref{eq:VPE_linear_system} of the VPE \eqref{eq:VHE_EBC}. For instance, the first-order scheme,
\begin{equation*}
\begin{cases}[r@{~}l@{\quad}l]
\tfrac{1}{\Delta t}\prn{\vec{u}(t+\Delta t)-\vec{u}(t)}
&= \nu\Delta\vec{u}(t+\Delta t)+\vec{f}(t) &\text{in~}\Omega \\
\nabla\cdot\vec{u}(t+\Delta t) &= 0 &\text{on~}\partial\Omega \\
\vec{n}\times\vec{u}(t+\Delta t)
&= \vec{n}\times\vec{g}(t+\Delta t) &\text{on~}\partial\Omega \\
\vec{u}(0) &= \mathring{\vec{u}} &\text{in~}\Omega\;,
\end{cases}
\end{equation*}
amounts to the following modifications of system \eqref{eq:VPE_linear_system}: (i) The vector of unknowns becomes the new vector field at time $t+\Delta t$. (ii) In the right hand side vector, the function $g$ is evaluated at time $t+\Delta t$. (iii) In the top two block rows of the system matrix, multiply all entries by $\nu$ and add $1/\Delta t$ to the diagonal entries. (iv) Add $1/\Delta t$ times the solution at time $t$ to the top two blocks of the right hand side vector. Everything else remains unchanged.

In the forward Euler case, the update at the interior points becomes explicit, while the boundary conditions are still implicit. Hence, to advance the solution from time $t$ to $t+\Delta t$, one first updates at each interior point explicitly
\begin{equation*}
\vec{u}(x_i,t+\Delta t) = \vec{u}(x_i,t)
+ \Delta t\,\nu\!\sum_{j\in B_i} a_{ij} \vec{u}(x_j,t) + \Delta t\,\vec{f}(x_i,t)
\end{equation*}
according to \eqref{eq:stencil_linear_combination} and \eqref{eq:stencil_WLSQ}, and after that solves a small linear system for the boundary points that results from the two bottom block rows of system \eqref{eq:VPE_linear_system}, where the left two block columns are brought to the right hand side (using the just updated interior point values).

\begin{figure}
\begin{minipage}[b]{.32\textwidth}
\includegraphics[width=\textwidth]{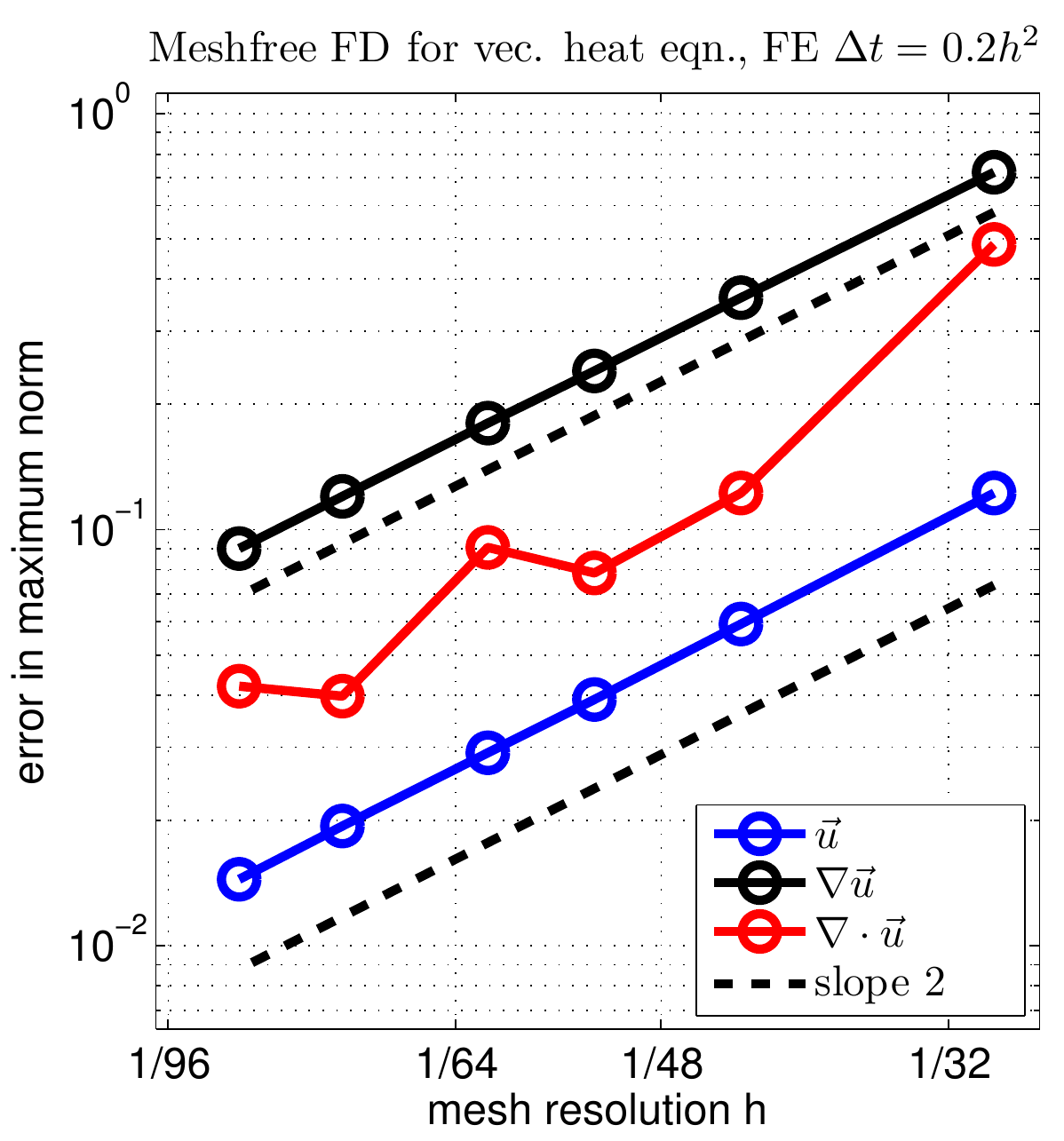}
\end{minipage}
\hfill
\begin{minipage}[b]{.32\textwidth}
\includegraphics[width=\textwidth]{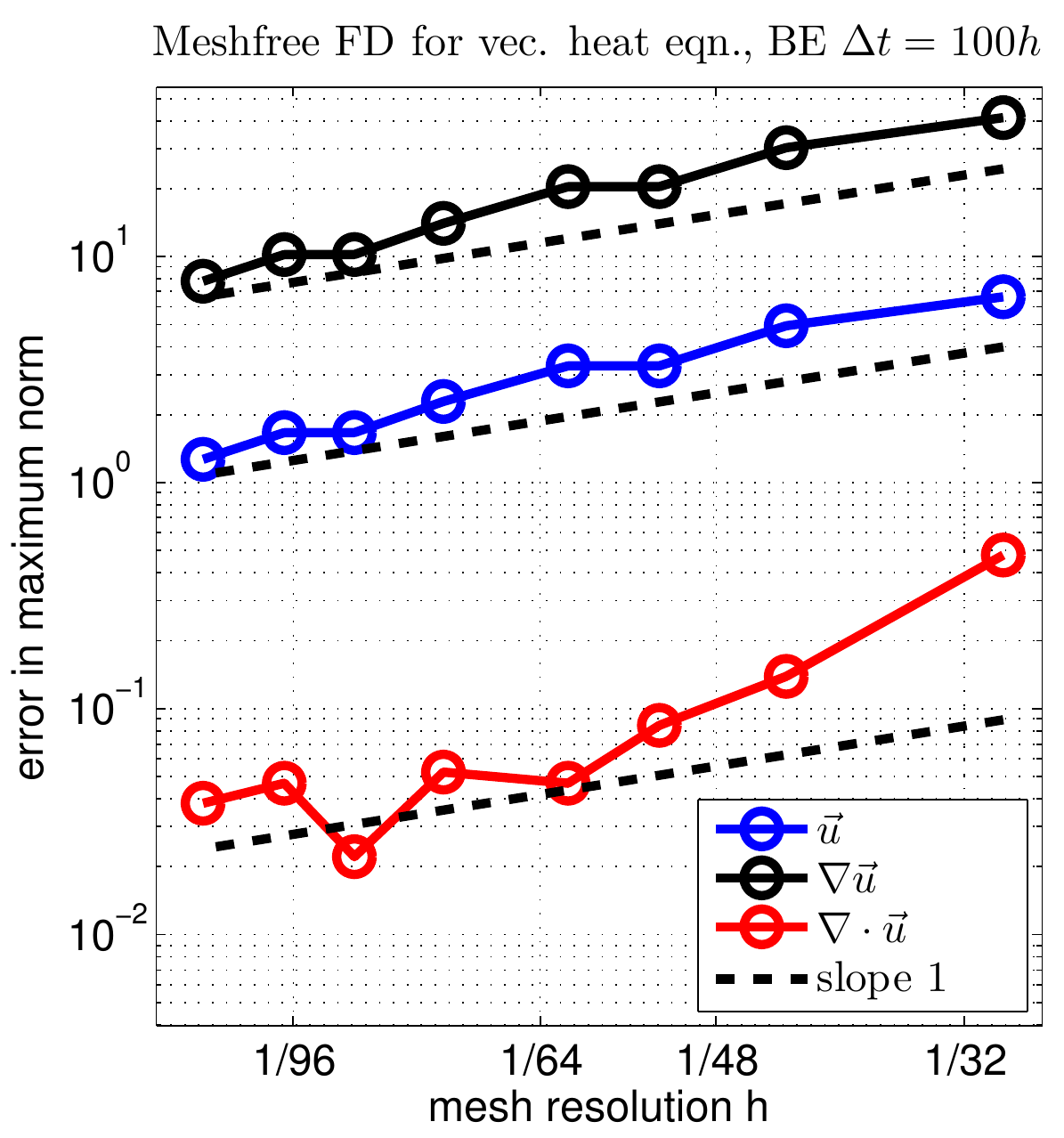}
\end{minipage}
\hfill
\begin{minipage}[b]{.32\textwidth}
\includegraphics[width=\textwidth]{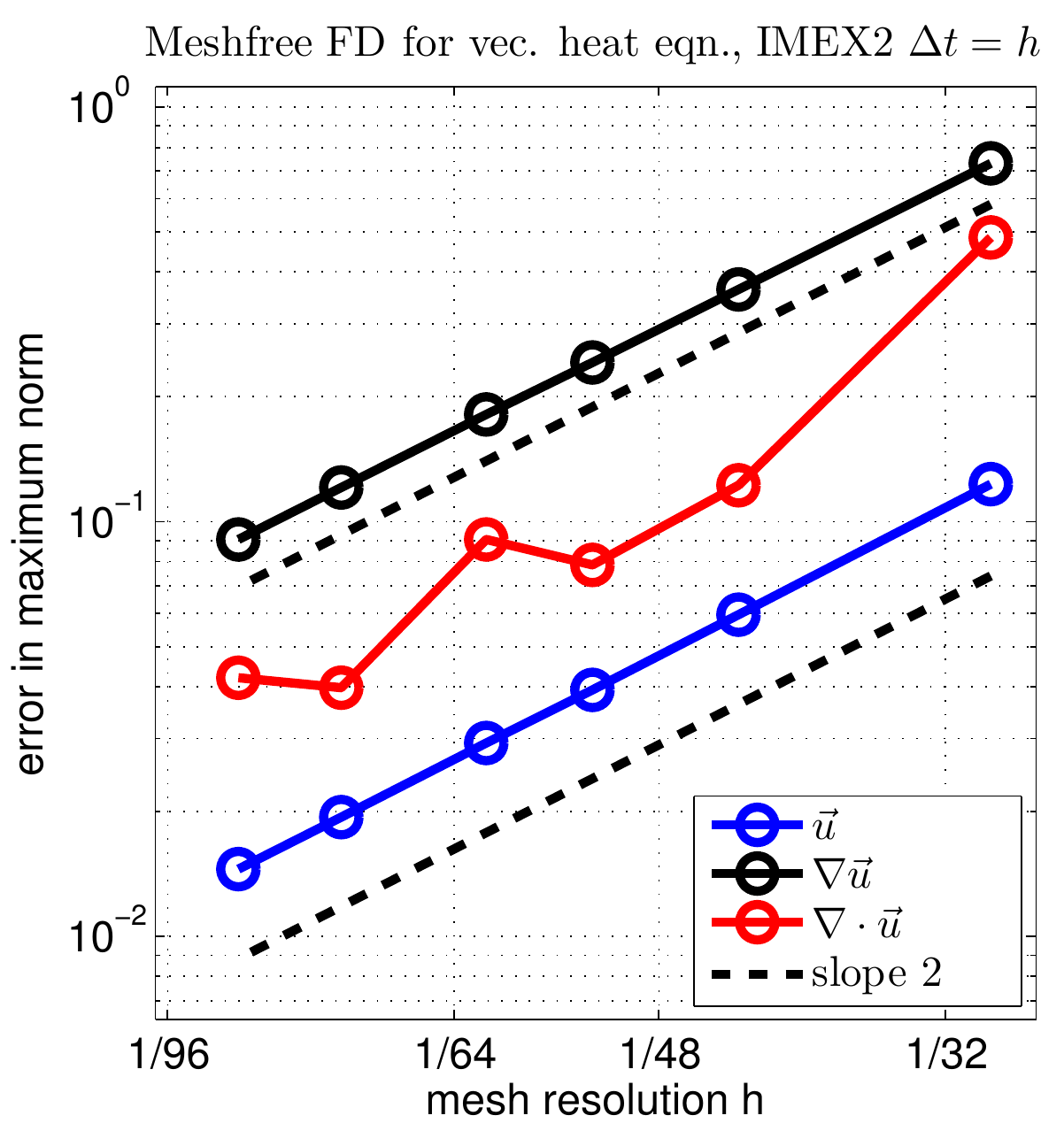}
\end{minipage}
\vspace{-.4em}
\caption{Error convergence for the vector heat equation, using a spatially second order meshfree finite difference discretization. Left panel: using forward Euler time stepping with $\Delta t = 0.2h^2$, confirming the expected $O(h^2)$ convergence order. Middle panel: using backward Euler time stepping with $\Delta t = 100h$, yielding the expected $O(h)$ convergence due to temporal errors. Right panel: using a second-order ImEx scheme with $\Delta t = h$, confirming the expected $O(h^2)$ convergence order.}
\label{fig:VHE_meshfree_solutions}
\end{figure}

To study these numerical schemes, we generalize the manufactured solution from \S\ref{subsubsec:FEM_manufactured_solution} to the time-dependent case. We set $\nu=1$, prescribe the incompressible solution
\begin{equation}
\label{eq:manufactured_solution_time}
\vec{u}(x,y,t) =
\begin{pmatrix}
\phantom{-}\pi\cos(t)\sin(2\pi y)\sin^2(\pi x) \\
-\pi\cos(t)\sin(2\pi x)\sin^2(\pi y)
\end{pmatrix}\,,
\end{equation}
and calculate the forcing $\vec{f} = \partial_t\vec{u}-\nu\Delta\vec{u}$, the boundary velocity $\vec{g} = \vec{u}$, and the initial conditions $\mathring{\vec{u}} = \vec{u}(t=0)$ accordingly. Using this test case, we first determine, via numerical experiments, the maximum time step that the forward Euler scheme admits to be stable. We find $\Delta t\le \frac{Ch^2}{\nu}$, where for the given point clouds, $C$ always lies between 0.2 and 0.3. Then, we study the convergence orders of the numerical schemes. The results are shown in Fig.~\ref{fig:VHE_meshfree_solutions}. We use a second order meshfree FD approximation in space, and conduct five kinds of time stepping: forward Euler and backward Euler with $\Delta t = 0.2h^2$ (left panel; the plots of forward vs.\ backward Euler are indistinguishable); backward Euler with $\Delta t = h$ (not shown; due to very small temporal errors, the convergence looks like second order); backward Euler with $\Delta t = 100h$ (middle panel; the temporal error is visible and yields the expected drop to first order); and the second-order ImEx scheme \eqref{eq:IMEX_second} (right panel). From these results, we can see that the solution, and its derivatives, are in fact second-order accurate in space. Moreover, explicit and implicit time stepping (of first and second order) can be conducted without problems; and the temporal errors are relatively small.

\subsection{Meshfree Finite Differences for the PPE Reformulation}
\label{subsec:PPE_meshfree_FD}
Structurally the PPE reformulation \eqref{eq:NSE_EBC} is the same as the vector heat equation \eqref{eq:VHE_EBC}, ``just'' with the nonlinear term $\vec{N}(\vec{u}) = (\vec{u}\cdot\nabla)\vec{u}$ and the pressure term $\nabla P(\vec{u})$ added to the time evolution. We treat both of these terms, as well as the forcing, explicitly. The first-order ImEx time-stepping \eqref{eq:IMEX_first} gives rise to the update rule
\begin{equation}
\label{eq:PPE_time_stepping}
\tfrac{1}{\Delta t}\prn{\vec{u}(t+\Delta t)-\vec{u}(t)}
= -\vec{N}(\vec{u}(t)) - \nabla P(\vec{u}(t)) + \nu\Delta\vec{u}(t+\theta\Delta t) + \vec{f}(t)\;,
\end{equation}
where $\theta\in\{0,1\}$ allows to switch between an explicit/implicit treatment of viscosity. In the forward Euler case ($\theta=0$), stability requires $\Delta t = O(h^2)$, and thus the scheme's accuracy is $O(h^2)$ overall. In the semi-implicit case ($\theta=1$), one can choose larger time steps, and consequently the temporal accuracy is not sufficient. We therefore use instead the second-order ImEx time-stepping \eqref{eq:IMEX_second} with $\vec{R}(\vec{u}) = \nu\Delta\vec{u}$ and $\vec{Q}(\vec{u}) = \vec{f}-\vec{N}(\vec{u})-\nabla P(\vec{u})$, which allows for time steps $\Delta t = O(h)$ and yields an $O(h^2)$ accurate scheme.

The Jacobi matrix $\nabla\vec{u}$ needed for the nonlinear terms is approximated very simply via point-centered meshfree finite differences, via the methodology described in \S\ref{subsec:VPE_meshfree_FD}. Clearly, such a centered treatment of advection is not the most effective choice for high Reynolds numbers (i.e., $\nu\ll 1$). And in fact, meshfree FD are quite easily amenable to an upwind treatment (e.g., by centering the approximation around a position $x_i-\beta\vec{u}(x_i)$, where $\beta$ is a suitably chosen parameter). However, for the purpose of demonstrating the convergence of meshfree FD methods for the PPE reformulation \eqref{eq:NSE_EBC}, the central treatment of $\vec{N}(\vec{u})$ is sufficient.

The pressure $P(\vec{u})$ results from the solution of the pressure Poisson equation \eqref{eq:PPE_EBC}. We discretize this problem via the same meshfree FD method described in \S\ref{subsec:VPE_meshfree_FD}, with one important deviation from the standard procedure. The right hand side of the boundary conditions in \eqref{eq:PPE_EBC} requires the evaluation of $\Delta\vec{u}$ at the boundary $\partial\Omega$. While straightforward meshfree FD for $\Delta\vec{u}$ yield accurate approximations inside the domain, it turns out that low accuracy (in the form of bounded but noticeable spatial oscillations along $\partial\Omega$) is achieved when using the same procedure at a boundary point. The reason is that the Laplacian is an operator that naturally ``likes'' to use data around the approximation point (cf.~\cite{Seibold2008}); however, at the boundary, data in such a configuration is not accessible. We therefore employ a different approach that remedies the problem: we use the meshfree approximation of $\vec{w} = \Delta\vec{u}$ at the interior points (as calculated for the viscosity), and extrapolate this field $\vec{w}$ to the boundary points, using moving least squares (MLS) interpolation \cite{LancasterSalkauskas1981}. This aspect is visualized in Fig.~\ref{fig:Laplacian_error_function}: the black dots are the approximation errors when approximating $\Delta\vec{u}$ at $\partial\Omega$ via meshfree FD; the red dots are the errors obtained when using MLS interpolation.

Since \eqref{eq:PPE_EBC} is a Neumann problem, its discretization leads to a linear system $A\cdot p = r$, in which the Poisson matrix $A$ has corank 1. In fact, because in the meshfree FD expansion \eqref{eq:stencil_linear_combination} the first term must vanish for any differential operator, the kernel of $A$ is $e = (1,\dots,1)^T$. In line with the approach described in \cite{Henshaw1994}, we solve the augmented system
\begin{equation}
\label{eq:augmented_pressure_system}
\begin{pmatrix} A & e \\ e^{T} & 0 \end{pmatrix}
\cdot \begin{pmatrix} p \\ \alpha \end{pmatrix}
= \begin{pmatrix} r \\ 0 \end{pmatrix}\;,
\end{equation}
whose unique solution is the one satisfying $A\cdot p = r-\alpha e$, where the new right hand side is the projection of $r$ onto the range of $A$. Moreover, a unique solution is singled out by the condition $e^T\cdot p = 0$. This approach in particular addresses the possibility that the PPE compatibility condition \eqref{eq:PPE_compatibility_condition} may be not satisfied exactly. The gradient of the resulting pressure is then approximated via standard meshfree FD at all interior points.

\begin{figure}
\begin{minipage}[t]{.32\textwidth}
\includegraphics[width=1.01\textwidth]{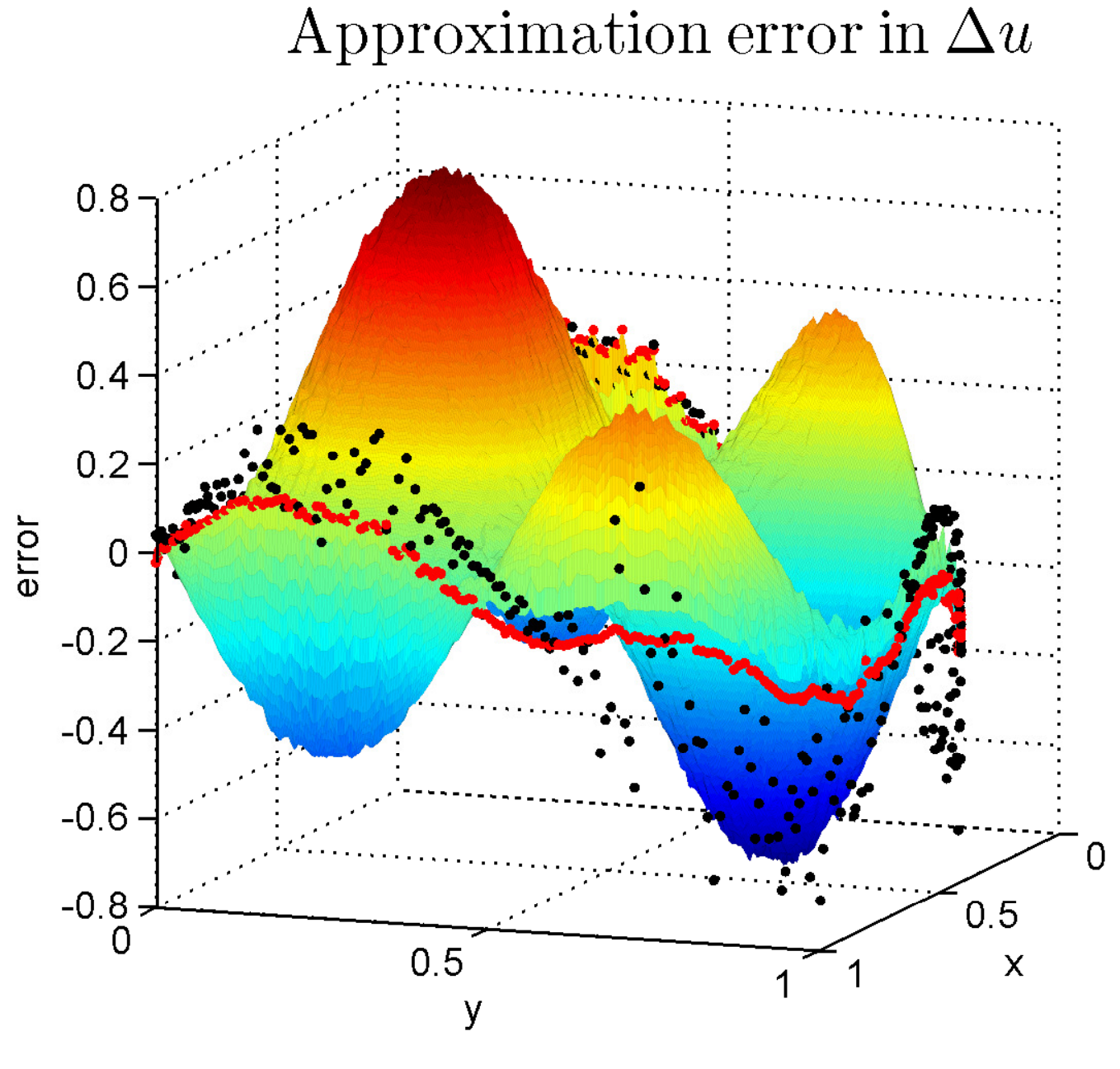}
\vspace{-1.4em}
\caption{Error function when approximating $\Delta u$ using meshfree FD. Black dots: use FD directly at boundary points. Red dots: extrapolate function $\Delta u$ from interior to the boundary.}
\label{fig:Laplacian_error_function}
\end{minipage}
\hfill
\begin{minipage}[t]{.63\textwidth}
\includegraphics[width=.495\textwidth]{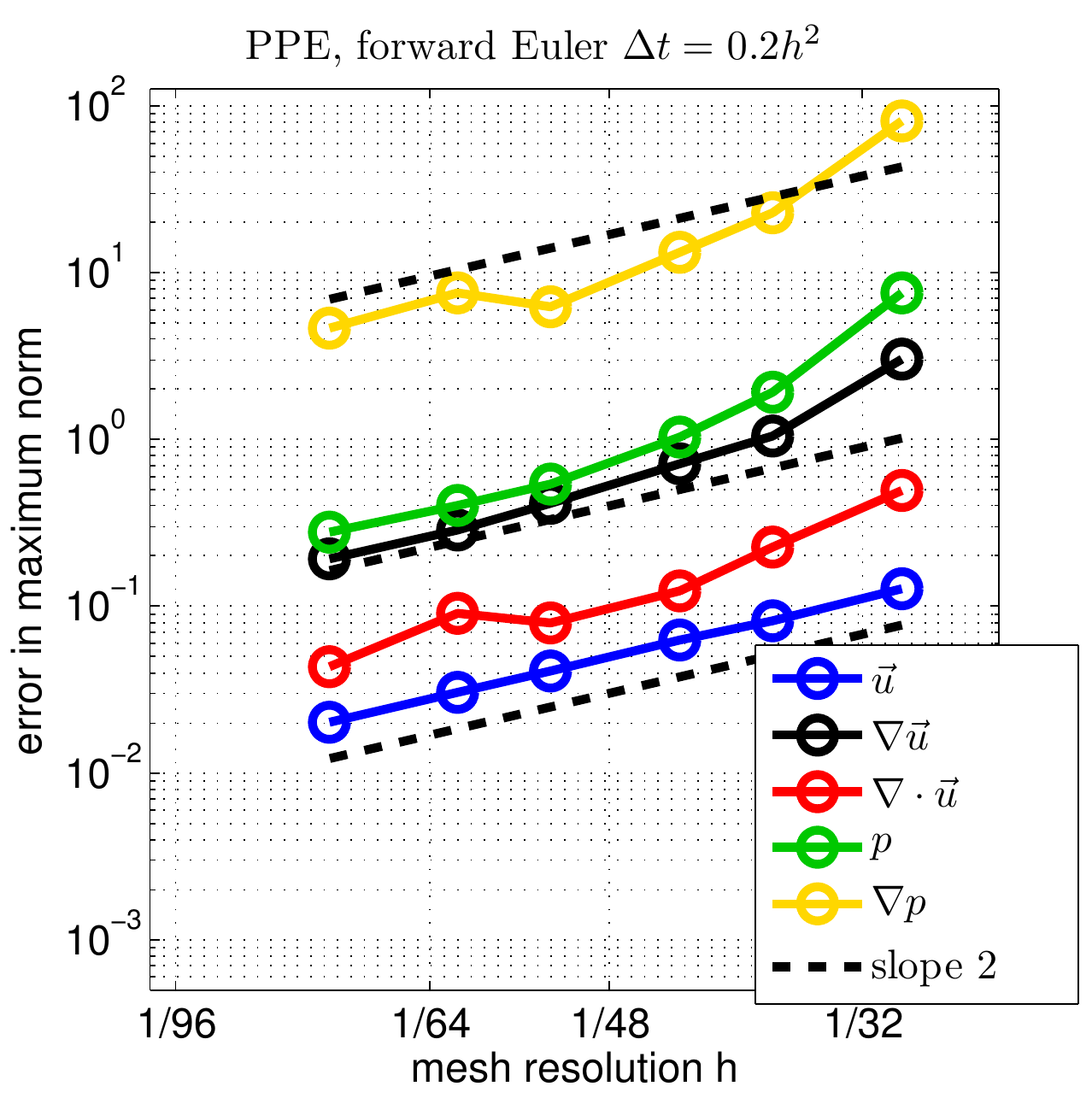}
\hfill
\includegraphics[width=.495\textwidth]{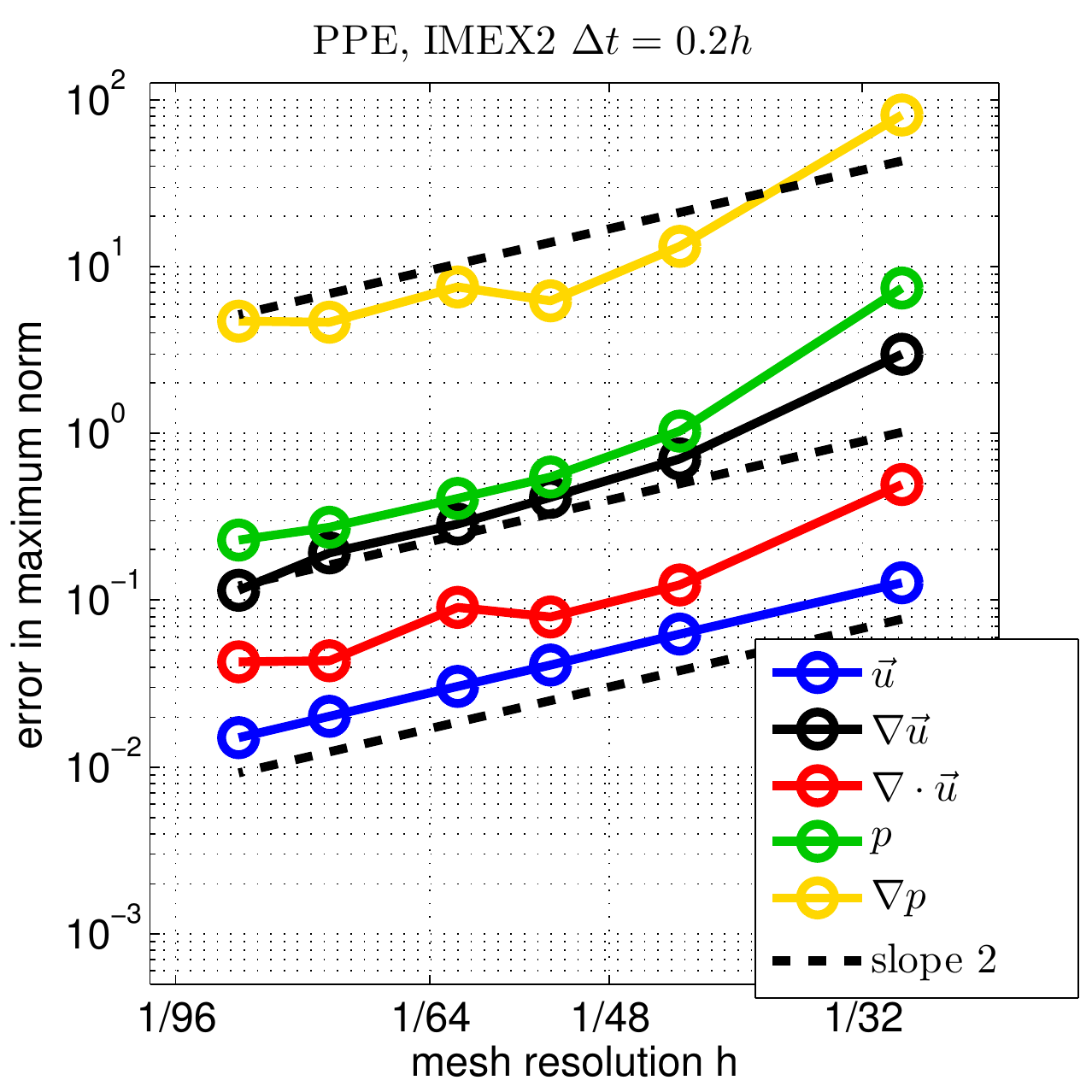}
\vspace{-1.4em}
\caption{Error convergence of the velocity field, the pressure, and their derivatives, for the PPE reformulation of the NSE. A spatially second order meshfree FD discretization is used. Left panel: forward Euler time stepping (with $\Delta t = 0.2h^2$). Right panel: ImEx2 time stepping (with $\Delta t = 0.2h$). In both cases, the expected $O(h^2)$ convergence is confirmed.}
\label{fig:PPE_meshfree_solutions}
\end{minipage}
\end{figure}

The maximum admissible time step of the numerical scheme is determined by the viscosity term in the explicit case, i.e., $\Delta t = O(h^2)$, and by the nonlinear advection term in the semi-implicit case, i.e., $\Delta t = O(h)$. These stability time step restrictions are in line with those observed for the numerical method presented in \cite{ShirokoffRosales2010}. However, they are different from the phenomenon observed and analyzed in \cite{Petersson2001, HenshawPetersson2003} for a different PPE reformulation and a different numerical discretization. In that study, the parabolic scaling $\Delta t = O(h^2)$ is observed to be required for stability, even if viscosity is treated implicitly.

\subsection{Numerical Results}
\label{subsec:PPE_numerical_results}
In order to investigate the convergence of the numerical scheme developed in \S\ref{subsec:PPE_meshfree_FD}, we use the same manufactured solution \eqref{eq:manufactured_solution_time} as for the VHE, set $\nu = 1$, and calculate the pressure $p(x,y,t) = -\cos(t)\cos(\pi x)\sin(\pi y)$, the forcing $\vec{f} = \partial_t\vec{u}+(\vec{u}\cdot\nabla)\vec{u}+\nabla P(\vec{u})-\nu\Delta\vec{u}$, and the boundary velocity $\vec{g} = \vec{u}$ accordingly. We use a spatially second order meshfree FD scheme (with the special treatment of $\Delta\vec{u}|_{\partial\Omega}$, see \S\ref{subsec:PPE_meshfree_FD}), and two types of time stepping: a) forward Euler with $\Delta t = 0.2h^2$; and second-order ImEx with $\Delta t = 0.2h$. In all cases the boundary relaxation value is chosen $\lambda=30$. The numerical results, shown in Fig.~\ref{fig:PPE_meshfree_solutions}, demonstrate that we obtain an overall second order convergence rate for all quantities of interest: the velocity field, its gradient, the divergence, the pressure, and the pressure gradient.

\begin{figure}
\begin{minipage}[b]{.485\textwidth}
\includegraphics[width=\textwidth]{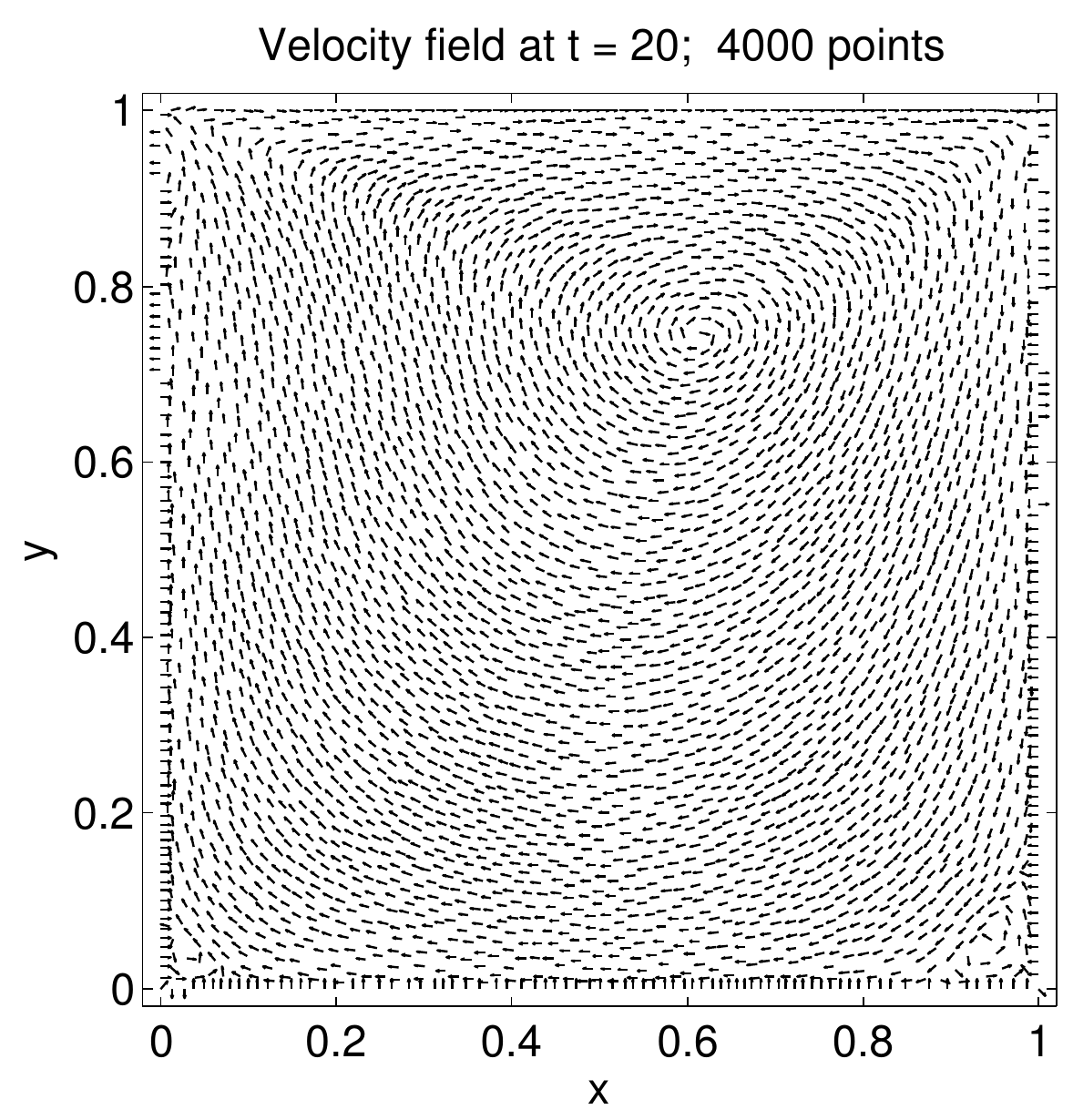}
\end{minipage}
\hfill
\begin{minipage}[b]{.495\textwidth}
\includegraphics[width=\textwidth]{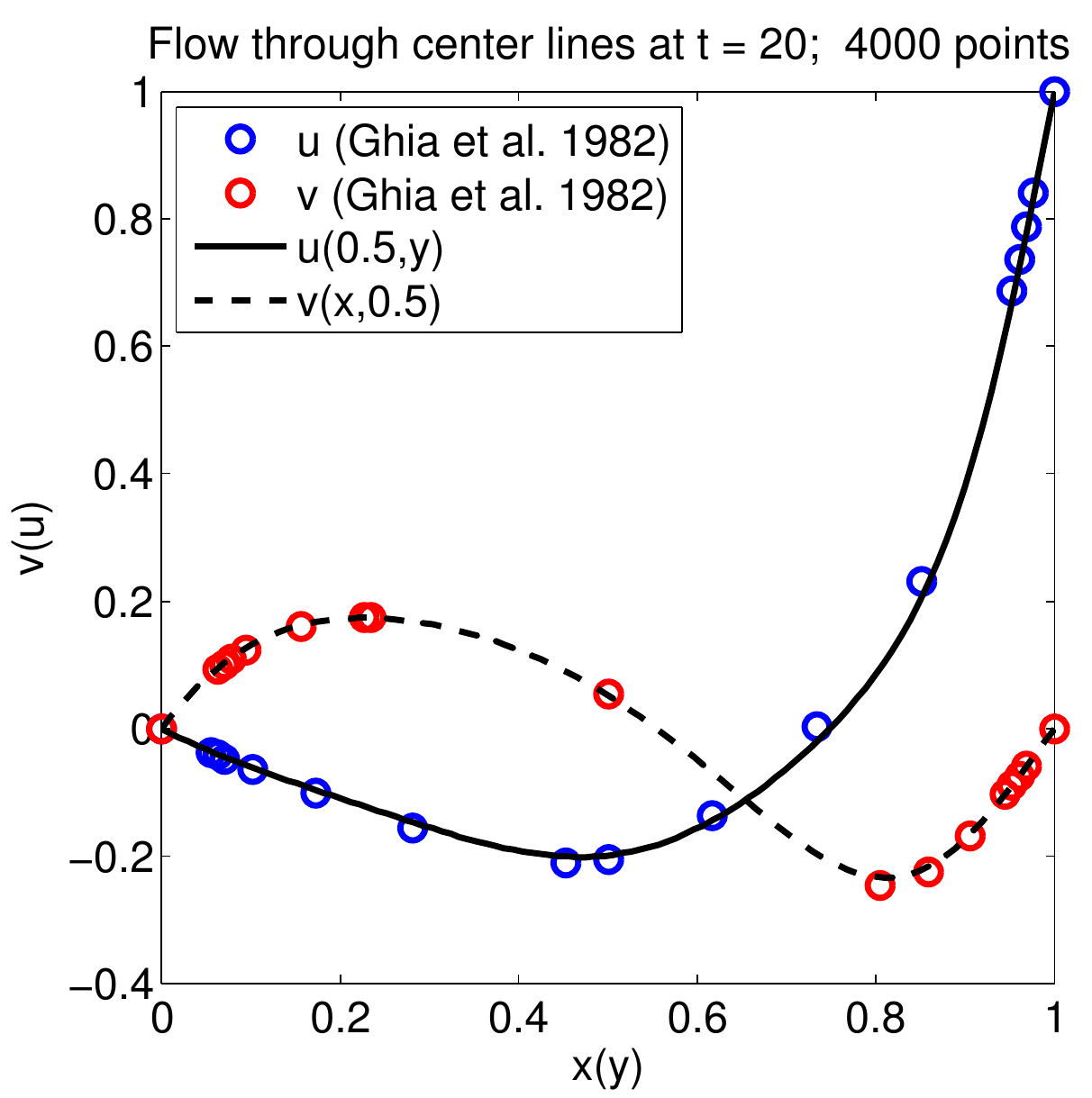}
\end{minipage}
\vspace{-.4em}
\caption{Lid-driven cavity test with Re=100 using a second-order meshfree scheme with 4000 points. Left: normalized velocity field at time $T=20$. Right: plot of the flow normal to the two centerlines of the cavity compared with reference data \cite{GhiaGhiaShin1982}.}
\label{fig:PPE_lid_driven_cavity}
\end{figure}

Moreover, to demonstrate the applicability of the numerical methodology, we conduct the standard benchmark lid-driven cavity test \cite{Burggraf1966} for Reynolds number 100, i.e., $\nu=0.01$. On the domain $\Omega = (0,1)^2$, the velocity field is zero at the boundaries, except for the tangential velocity at $y=1$, which is 1. The initial velocity field is zero everywhere except for $y=1$, where it equals the boundary condition. The numerical approach used here is the same forward Euler-based scheme as in the manufactured solution test. The results of the steady-state profile ($t=20$) are shown in Fig.~\ref{fig:PPE_lid_driven_cavity}. The velocity field (left panel) is depicted in the form of normalized vectors at the approximation points. The large center vortex and the two vortices in the bottom corners are captured. Moreover, a comparison (right panel) of the velocity through the centerlines of the cavity with reference data \cite{GhiaGhiaShin1982} shows a good agreement even on this not very highly resolved point cloud. Note that one particularity of the PPE reformulation \eqref{eq:NSE_EBC} is that the flow through the boundary need not necessarily be exactly zero due to numerical approximation errors. This is why a flow through the boundaries is visible in the scaled quiver plot. However, the actual flow through the boundary is very small.

\vspace{1.5em}
\section{Conclusions and Outlook}
\label{sec:conclusions_outlook}
The results in this paper show that meshfree finite differences (FD) provide a relatively straightforward methodology to approximate the solutions of vector-valued elliptic, parabolic, and fluid flow problems with electric boundary conditions (EBC), on domains without re-entrant corners. This is in contrast to finite element methods (FEM), whose simplest version, nodal-based FEM on triangular elements, fails at generating the correct solution. Instead, a Babu{\v s}ka paradox arises, which is shown to arise from the fact that nodal-based FEM do not capture the domain boundary's curvature.

For the vector Poisson equation, meshfree FD lead to a linear system that discretizes the Laplace operator at interior points, and the divergence operator at boundary points, in a natural and very systematic fashion. The same methodology is shown to yield first, second, and third order convergent numerical schemes. Analogous statements hold for the vector-heat equation. Implicit time stepping is a straightforward extension of the vector Poisson case; explicit time stepping is a bit more interesting (because boundary conditions remain implicit), but poses no conceptual complication.

The extension of the methods to a PPE reformulation of the Navier-Stokes equation with EBC is, again, conceptually not complicated. There is one challenge that must be overcome, namely the approximation of the Laplacian of the velocity field at the domain boundary. Once this issue has been addressed, a second-order accurate numerical scheme is obtained in which the pressure solve and the viscosity solve are decoupled. As a consequence, one can choose between an explicit and an implicit treatment of viscosity.

Being FD approaches, the numerical schemes yield the values of the velocity field $\vec{u}$ at the approximation points only. However, meshfree stencils can be employed to also calculate spatial derivatives of the velocity field. For instance, the velocity gradient $\nabla\vec{u}$ is crucial in computing forces and stresses acting on the boundary. Our investigation of the accuracy of these derivative quantities reveals that, for all studied problems, they show no degradation in order: a $k^\text{th}$ order scheme yields $k^\text{th}$ order convergence in $\vec{u}$, and also in $\nabla\vec{u}$.

While the results demonstrate the potential of meshfree FD for these types of problems, they also give rise to further questions. One important question is whether domains with re-entrant corners can also be treated. At first glance, one would think ``no'', because of the lack of smoothness of the solutions on such domains. However, FD methods are known to be able to yield correct answers even for certain problems that lack smoothness (such as hyperbolic conservation laws \cite{LeVeque1992}). This work also gives rise to a number of questions regarding the high-order accurate meshfree FD approximations. First, extensions of the numerical schemes for the PPE reformulation to convergence orders higher than two are of interest. Second, to avoid excessively large stencils for higher approximation orders, it is of interest whether the numerical schemes would work equally well (or better) if the meshfree FD approximations were obtained in a different fashion, such as via radial basis functions, compact FD, or deferred correction.

\vspace{1.0em}
\section*{Acknowledgments}
P. Chidyagwai, R. R. Rosales, B. Seibold, and D. Zhou wish to acknowledge support by the National Science Foundation. This work was supported through grants DMS--1115269 and DMS-1115278. Furthermore, R. R. Rosales would like to acknowledge partial support by NSF through grant DMS--1318942, and B. Seibold would like to acknowledge partial support by NSF through grants DMS--1318641 and DMS--1318709. D. Shirokoff acknowledges partial funding by NSERC.

\bibliographystyle{plain}
\bibliography{references_complete}

\end{document}